\documentclass[10pt]{amsart}
\usepackage{amsmath}
\usepackage{amsfonts}
\usepackage{amssymb}
\usepackage{graphicx}
\usepackage{amsthm,graphicx,color,yfonts}
\usepackage{pdfsync}
\usepackage{epstopdf}%

\usepackage[colorlinks=true]{hyperref}
\hypersetup{linkcolor=black,citecolor=black,filecolor=black,urlcolor=black} 

\providecommand{\U}[1]{\protect\rule{.1in}{.1in}}

\newcommand{\C}{{\mathbb C}}
\newcommand{\R} {\mathbb R}
\newcommand{\cuad}{{\sqcap\kern-.68em\sqcup}}

\newcommand{\ve}{\varepsilon}

\newcommand{\be}{\begin{equation}}
\newcommand{\ee}{\end{equation}}

\newcommand{\st}{{\tt t}}
\newcommand{\sy}{{\tt y}}
\newcommand{\pD}{{\mathring D_\tau}}

\definecolor{darkgreen}{rgb}{0.2,0.7,0.1}

\newcommand{\blue}{\color{black}}

\newcommand*{\ep}{\varepsilon}

\newcommand{\modulo}{{\,\mathrm{mod}\,}}

\numberwithin{equation}{section}

\newtheorem{theorem}{Theorem}[section]
\newtheorem{proposition}{Proposition}[section]

\newtheorem{lemma}{Lemma}[section]

\newtheorem{remark}{Remark}[section]

\title[Jacobi fields of  rotationally symmetric  solutions to the Cahn-Hillard equation]{Nondegeneracy and the Jacobi fields of  rotationally symmetric  solutions to the Cahn-Hillard equation}

\author{\'Alvaro Hern\'andez}
\address{Universidad de Los Andes, Facultad de Ingenier\'ia y Ciencias Aplicadas. }
\email{ahernandez2@miuandes.cl}

\author{Micha{\l } Kowalczyk}
\address{Departamento de Ingenier\'{\i}a Matem\'atica and Centro
de Modelamiento Matem\'atico (UMI 2807 CNRS), Universidad de Chile.}
\email {kowalczy@dim.uchile.cl}
\thanks{M. Kowalczyk was partially supported by Chilean research grants Fondecyt 1130126, 1170164 and Fondo Basal CMM-Chile}


\begin{document}

\maketitle

\begin{abstract}
In this paper we study  rotationally  symmetric solutions of the Cahn-Hilliard equation in $\R^3$ constructed in \cite{ch_arxiv} by the authors. These solutions form a one parameter family analog to the family of  Delaunay surfaces and in fact the zero level sets of their blowdowns approach these surfaces. Presently we go a step further and show that  their stability properties are  inherited from the stability properties of the Delaunay surfaces. Our main result states that the rotationally symmetric solutions are non degenerate and that they have exactly $6$ Jacobi fields of temperate growth coming from the natural invariances of the problem (3 translations and 2 rotations) and the variation of the Delaunay parameter. 
\end{abstract}

\section{Introduction} 
\subsection{Statement of the main result}

In the classical van der Waals-Cahn-Hilliard  theory that describes the process of phase separation of two components of a binary  alloy one considers 
the Helmholtz free energy functional 
\begin{equation}\label{helmholtz energy}
 E_{{\ep}}(u)=\int_\Omega \left(F(u(x))+\frac12\ep^2|\nabla
u(x)|^2\right)\ dx
\end{equation}
in $H^{-1}(\Omega)$ subject to the average concentration to be constant, i.e. 
\begin{equation}
\label{mass constr}
 \frac1{|\Omega|}\int_\Omega u\ dx =m, 
\end{equation} 
where $m\in[-1,1]$ (see \cite{MR1772733}, \cite{MR1901064} for details) and  
the double-well potential $F(u)$ corresponds to the free energy density at low temperatures, which this paper we will take explicitly 
\begin{equation*}
 F(u)=\frac14\left(1-u^2\right)^2,\quad F'(u)=u^3-u.
\end{equation*}
From now on we will denote $F'(u)=-f(u)$. Note that  constant functions $u\equiv \pm 1$ are minimizers of this functional subject to  $m=\pm 1$. 
The  Euler-Lagrange equation  (with $f(u)=-F'(u)$) is 
\begin{equation}\label{CH stat}
\begin{aligned}
\ep^2\Delta u+f(u)=\delta_\ep\qquad &\text{ in }\Omega,\\
\frac{\partial u}{\partial \nu}=0\qquad &\text{ on }\partial\Omega,\\
\frac1{|\Omega|}\int_\Omega u\ dx=m
 \end{aligned}
\end{equation}
where $\delta_\ep$ is a Lagrange multiplier.

 Using $\Gamma$-convergence approach Modica \cite{MR866718}  showed that  minimizers $u_\ep$ of \eqref{helmholtz energy} subject to constraint (\ref{mass constr}) 
 $\Gamma$-converge, {as $\ep\to 0$}, to the function $1-2\chi_{A_0}$, where
$\chi_{A_0}$ is the characteristic function of an open set 
$A_0\subset\Omega$. Moreover $\partial A_0\cap \Omega$ is locally a surface of  constant mean curvature (CMC surface for short). Geometrically the set $A_0$ minimizes the perimeter functional $\mathrm{Per}_\Omega(A)$ among the sets $A\subset \Omega$ whose volume is fixed.  A generalisation of these results was given by Sternberg \cite{MR930124}.  Furthermore Hutchinson and Tonegawa \cite{MR1803974} studied limits of general critical points \eqref{helmholtz energy} and showed that their limits are locally minimal or CMC surfaces. On the other hand it  is known \cite{MR985990} that if a set $A\subset \Omega$  is an {\it isolated} mimimizer of the perimeter functional subject to the constant volume constraint then there exists a sequence of minimizers $u_\ep$ of (\ref{helmholtz energy}) which $\Gamma$ converges to $A$. This result can be used to construct solutions to (\ref{CH stat}) at least in dimension $2$ , see \cite{MR1399196}. The most complete construction is due to  Pacard and Ritor\'e \cite{MR2032110} 
who proved the following:  if $M$ is a  compact Riemannian manifold and $N$ is a non degenerate minimal or CMC sub manifold of $M$ which divides $M$ into $2$ disjoint components then for all sufficiently small  $\ve$ there exist critical points of \eqref{helmholtz energy} whose $0$ level set converges to $N$.  
The counterpart of this theory for the time dependent problem  was developed among others by Alikakos, Bates and Chen \cite{MR1308851} who proved that as $\ve\to 0$ the time evolution of interfaces is governed by the Helle-Shaw problem-of course CMC surfaces are stationary points of the flow. More detailed description of the Cahn-Hilliard flow and key spectral tools   can be found for instance in  \cite{MR1237062}, \cite{MR1382059}, \cite{MR1613496}, \cite{MR1797871}, \cite{MR1284813} and the references therein. Additional examples of stationary solutions for the singular perturbation problem  in a bounded domain have been constructed  in \cite{MR1632937}, \cite{MR1636688}, \cite{MR1737000}.

In this paper we consider  stationary solutions of the Cahn-Hilliard   in the whole space, namely solutions to  the following problem:
\begin{equation}
\label{CH1}
\Delta u+f(u)=\delta, \quad \mbox{in}\ \R^3,
\end{equation}
It is convenient to rescale the equation (\ref{CH1}) by dilation of the independent  variable by a (large) factor $\ep^{-1}>0$
\[
{\tt x}\longmapsto \ep^{-1}{\tt x},
\]
and  obtain an equivalent form of (\ref {CH1}):
\begin{equation}
\label{CH}
\ep\Delta u+\frac{1}{\ep}f(u)={{\ell}}_\ep, \quad \mbox{in}\
\R^{{^3}}.
\end{equation} 
where we have denoted $\frac{\delta}{\ep}={{\ell}}_\ep$. Clearly, if
$u_\ep$ is a solution of (\ref{CH}) then $v({\tt x})= u_\ep\big({\ep}{\tt
x}\big)$ is a solution of (\ref{CH1}). On the other hand, if $v$ is a solution
of (\ref{CH1}) then $u_\ep({\tt x})=v\big(\frac{\tt x}{\ep}\big)$ is a solution
of (\ref{CH}). In particular this means that while phase transition of the
solutions of (\ref{CH1}) are of order $1$, for the solutions of (\ref{CH}) they
are of order $\ep$. Thus the latter are more "concentrated" and are  blowdowns of the former. In the sequel we
will focus on on the form  (\ref{CH}) of the stationary Cahn-Hilliard equation.  From what we have said above about the
singular perturbation problem  it is clear that level sets of these solutions
should converge, as $\ep$ tends to $0$, to CMC surfaces in $\R^3$. 

Now we will describe a family of such solutions. Let $D_\tau$, $\tau (0,1)$ be a  Delaunay unduloid and  let  $N_\tau$  be its  normal  vector field. Without loss of generality we may assume that $D_\tau$ is normalised in such a way that its mean curvature is $1$. 
Let us notice that the surface $D_{\tau}$ divides the space into two disjoint  components $\Omega^\pm_\tau$, such that  $\R^3\setminus D_{\tau}={\Omega}_{\tau}^+\cup{\Omega}_{\tau}^-$, where  $N_\tau$ points towards ${\Omega}_{\tau}^+$. By changing the orientation of $D_\tau$ if necessary  we can  chose $N_\tau$ in such a way that $\Omega_\tau^+$ contains the axis of symmetry of the surface. 
The following result is proven in \cite{ch_arxiv}:
\begin{theorem}\label{teorema 1}
 For all $\tau\in(0,1)$ there exits $\ep_\tau>0$ such that for all $\ep\in (0, \ep_\tau)$  the problem 
\begin{equation}\label{CH ep}
\ep \Delta u+\frac{1}{\ep}f(u)={{\ell}}_\ep\quad\text{ in}\ \R^3
\end{equation} 
has a solution $w_{\tau}$, which is  one-periodic in the direction of the 
$z$-axis and rotationally symmetric with respect to rotations about the same
axis. As $\ep\to 0$ we have  ${{\ell}}_\ep =1+\mathcal O(\ep)$, and 
$w_{\tau}$  satisfies
\begin{align*}
w_{\tau}\to 1 \text{ as }\ep\to 0&\text{ in }{\Omega}_{\tau}^+,\\
w_{\tau}\to -1\text{ as }\ep\to 0&\text{ in }{\Omega}_{\tau}^-,
\end{align*} 
uniformly over compacts.  Moreover $w_\tau$ is differentiable as a function of the parameter $\tau$.
\end{theorem}

The solution described in the above theorem  can be translated in the direction of the coordinate axis and rotated about the $x$ and $y$ axis (accepting that the $z$ axis is its axis of the rotational symmetry). Additionally the parameter of the family $\tau\in (0,1)$ can be varied as well. The $5$ symmetries and $\partial_\tau w_\tau$ determine  $6$ Jacobi fields of the linearized operator
\begin{equation}
\label{jac 1}
L_{w_\tau}=\ve\Delta+\frac{1}{\ve} f'(w_\tau).
\end{equation}
We will call them the {\it geometric} Jacobi fields.  
It is natural to ask whether {\it all} Jacobi fields come  from these natural invariances. 
The answer is provided by our main result:
\begin{theorem}\label{main theorem}
\begin{itemize}
\item[(i)]
For all $\tau\in (0,1)$ and all small $\ve$ the operator $L_{w_\tau}$ is nondegenerate in the sense that  $\mathrm{Ker}\, L_{w_\tau}=0$. 
\item[(ii)]
There exists $a>0$ such that the linear subspace of solutions of  
\[
L_{w_\tau}\phi=0,
\]  
with temperate growth in the direction of the axis of rotation of $w_\tau$ i.e. such that 
\begin{equation}
\|\phi e^{\,-a|z|}\|_{L^2(\R^3)}<\infty,
\label{temperate a}
\end{equation}
has dimension  $6$ and coincides with the linear subspace of the geometric Jacobi fields. 
\end{itemize}
\end{theorem}

We will see that any geometric Jacobi field is either bounded or it grows linearly in the direction of the $z$ axis. Note however that our theorem does not exclude the possibility of existence of a solution of $L_{w_\tau}\phi=0$ such that $\phi$ satisfies (\ref{temperate a}) with some large value of $a$.  

To explain the  importance of this result let us go back to the construction of Pacard and Ritor\'e \cite{MR2032110}. They consider problem (\ref{CH}) but instead of the whole space on  a smooth,  compact manifold $M$. They assume that there exists  $N\subset M$ which is a smooth sub manifold of constant mean curvature such that it is the nodal set of  a smooth function on $M$ for which $0$ is a regular value. In particular it follows that $N$ divides $M$ into two disjoint components $M^\pm(N)$, similarly as $D_\tau$ divides $\R^3$. Furthermore it is assumed that $N$ is non degenerate in the sense that the kernel of the Jacobi operator of $N$
\[
\mathcal L_N=\Delta_N+|A_N|^2+\mathrm{Ric}_g (\nu_N, \nu_N)
\]
is empty. 
Under these hypothesis it is shown in \cite{MR2032110} that for any small $\ve$ there exists a solution of the Cahn-Hilliard equation (\ref{CH stat}) which converges uniformly to $\pm 1$ over compact subsets of $M^\pm (N)$.    Our existence result in Theorem \ref{teorema 1} also relies on the non degeneracy of the Delaunay surfaces, which in this case means that their Jacobi operator does not have kernel, and moreover it uses the fact that the Jacobi fields of these surfaces can be classified. Theorem \ref{main theorem} goes further since it provides a classification of the Jacobi fields of the family $w_\tau$ of  rotationally symmetric solutions of (\ref{CH}). This type of result is crucial if one wants to construct new solutions  to (\ref{CH1})  build upon more complicated CMC surfaces in $\R^3$, such as some of those constructed for instance in  \cite{MR1043269}, \cite{MR903742}, \cite{mpp},  \cite{maz_pac_pol},  \cite{MR1807955}, \cite{MR2194146}, \cite{MR2590386} (see also related construction in  \cite{Kowalczyk:2014sf} for the Allen-Cahn equation on the plane).

To explain this let us recall that a non compact, Alexandrov embedded, complete  CMC surface with finite topology outside of a compact set consists of finitely many half Delaunay surfaces (\cite{meeks_cmc}, \cite{meeks1987}, \cite{MR1010168}) called Delaunay ends. In addition if the number of ends of such surface is $k$ and this surface is non degenerate then set of nearby CMC surfaces is an analytic manifold of dimension $3k$. This  was proven by Kusner, Mazzeo and Pollack   in \cite{MR1371233}  and the argument of their paper is in many ways inspired by the similar result for the singular Yamabe problem \cite{MR1356375}.
One of the problems is to decide whether given CMC surface is non degenerate and this is rather difficult problem except for the Delaunay surface for which separation of variables and ODE methods can be used to prove non degeneracy (see also \cite{MR2255385}). Pushing these arguments further one can also classify Jacobi fields with temperate growth \cite{mpp} and show that all of them came from the natural invariances of the family of Delaunay  surfaces. Starting from non degenerate Delaunay surface with $k$ ends  one can built more complicated examples by gluing to it either an extra end or another non degenerate surface and thus obtain CMC surfaces with arbitrary many ends.   In some cases these new surfaces are also nondegenerate, see for instance  \cite{mpp},  \cite{maz_pac_pol}, \cite{MR2194146}, \cite{MR2590386}.

Theorem  \ref{main theorem} is the precise analog of the result proven in \cite{mpp} but in the case of the Cahn-Hilliard equation. Given what we said about the linear properties of the Delaunay surface its assertion is expected, which does not mean that the proof is equally obvious.  Certainly what needs to be done is to connect the stability properties of the Delaunay surface $D_\tau$ and the corresponding solution $w_\tau$ of (\ref{CH}) and this can be achieved by  expressing $w_\tau$ in the Fermi coordinates of $D_\tau$ (Section \ref{sec fermi}). While $w_\tau$ is localized near $D_\tau$ this kind of expression is only valid in a neighbourhood of the surface and this is what complicates the situation (see Section \ref{sec two end}).  In order to deal with this in this paper we  replace the operator $L_{w_\tau}$ with another operator ${\mathbb L}_{w_\tau}$ (Section \ref{sec 222}), which locally agrees with the original one but which is easier to analyze. Using  this idea in   Section \ref{section 3} we prove our theorem.  

In this paper $C, c$ will stand for generic positive constants, $\delta$ will be a small, $\ve$ independent constant and $\alpha\in (0,1)$ will be a constant as well. 

\section{Preliminaries}

\subsection{The surfaces of Delaunay}\label{sec delaunay}\setcounter{equation}{0} 
As we have seen, our results are based in great part and in some sense they parallel the  theory of CMC surfaces in $\R^3$ and because of this we begin by describing the basic geometric object in this paper which is the family of  Delaunay surfaces. The following is a summary of what can be found for instance in \cite{MR1941630} or \cite{MR1807955}.
The Delaunay unduloids $D_\tau$, $\tau\in (0, 1)$   are CMC surfaces of revolution in $\R^3$ and based on this one can easily  parametrize them. Indeed, such parametrization has form
\begin{equation}\label{dela 0}
R_\tau (z, \theta)=(\rho_\tau(z)\theta, z), \quad (z, \theta)\in \R\times S^{1},
\end{equation} 
where 
\[
\frac{2\rho_\tau}{\sqrt{1+(\partial_z\rho_{t})^2}}-\rho_\tau^2=\tau^2, \quad \rho_\tau(0)=1-\sqrt{1-\tau^2}.
\]

We can   "normalize'' the Delaunay surface and suppose that the mean
curvature of $D_\tau$ is $1$ for all $\tau\in (0, 1)$.  
A convenient way to parametrize Delaunay unduloids is to use the  isothermal coordinates:
\begin{equation}
\label{dela 1}
X_\tau(s, \theta)=\frac{1}{2}(\tau e^{\,\sigma_\tau(s)}\theta, \kappa_\tau(s)), \quad (s, \theta)\in \R\times S^{1},
\end{equation}
where functions $(\sigma_\tau, \kappa_\tau)$ are the unique solutions of the following system of ODEs:
\begin{equation}
\label{dela 2}
\begin{aligned}
(\partial_s\sigma_\tau)^2+\tau^2\cosh^2\sigma_\tau=1, &\quad \partial_s\sigma_\tau(0)=0, \quad \sigma_\tau(0)<0, \\
\partial_s\kappa_\tau-\tau^2\cosh^2{\,\sigma_\tau}=0, &\quad \kappa_\tau(0)=0.
\end{aligned}
\end{equation}
We will now summarize some basic facts about the Delaunay surfaces and their isothermal parametrization (we reproduce here as well as elsewhere in this section the results proven  in  \cite{mpp}, \cite{MR1356375}). We note first of all that  $\sigma_\tau$ is periodic, and consequently the surfaces $D_\tau$ are one-periodic along the $z$-axis: namely, if $T_\tau$ denotes the minimal  period then
\[
D_\tau=D_\tau+T_\tau {\tt e}_3.
\]
Clearly we have the relation 
\[
T_\tau=\frac{1}{2} \kappa_\tau(s_\tau),
\]
where $s_\tau$ is the minimal period of $\sigma_\tau$. 

The Jacobi operator $\mathcal J_{D_\tau}$ of $D_\tau$ is defined by:
\begin{equation}
\mathcal J_{D_\tau}:=\Delta_{D_\tau}+|A_\tau|^2,
\label{def jacobi}
\end{equation}
where $\Delta_{D_\tau}$ is the Laplace-Beltrami operator on $D_\tau$ and  $|A_\tau|^2$ is the square of the norm of the second fundamental form of $D_\tau$.  In the isothermal coordinates $(s, \theta)\in \R\times S^1$ its expression is given by:
\begin{equation}
\label{jacob 1}
{\mathcal J}_{D_\tau}=\frac{1}{\tau^2 e^{\,2\sigma_\tau}}\left\{\partial_s^2+\partial_\theta^2+\tau^2\cosh (2\sigma_\tau)\right\}.
\end{equation}
The geometric Jacobi fields on $D_\tau$  solve   $\mathcal J_{D_\tau}\Phi =0$ and are of three types:
\begin{itemize}
\item[(1)] {\it The Jacobi fields arising from  infinitesimal translations.} For any  ${\tt e}\in \R^3$, $|{\tt e}|=1$ we define:
\[
\Phi_\tau^{T, {\tt e}}={\tt e}\cdot N_\tau,
\]
where $N_\tau$ is the unit normal vector to $D_\tau$. The coordinate vectors ${\tt e}_j$, $j=1,2,3$ generate three linearly independent Jacobi fields $\Phi^{T, {\tt e}_j}_\tau$ corresponding to translations of $D_\tau$ in the directions of the coordinate axis. 
We note that in the isothermal coordinates  
\[
\Phi_\tau^{T, {\tt e}_3}=\Phi_\tau^{T, {\tt e}_3}(s), \quad \Phi_\tau^{T, {\tt e}_j}=\Phi_\tau^{T, {\tt e}_j}(s, \theta), \quad j=1,2.
\]
It is important to notice that the Jacobi fields  $\Phi_\tau^{T, {\tt e}_j}$ are bounded.
\item[(2)]{\it The Jacobi fields arising from infinitesimal  rotations.} 
Let ${\tt e}\in \R^3$, $|{\tt e}|=1$ be such that ${\tt e}\cdot {\tt e}_3=0$. The Killing vector field corresponding to the rotation about the vector ${\tt e}$ is:
\[
{\tt x}\longmapsto ({\tt x}\cdot{\tt e}){\tt e}_3-({\tt x}\cdot{\tt e}_3){\tt e}.
\] 
We define the Jacobi field associated to this vector field by:
\[
\Phi_\tau^{R, {\tt e}}=\left[({\tt x}\cdot{\tt e}){\tt e}_3-({\tt x}\cdot{\tt e}_3){\tt e}\right]\cdot N_\tau.
\]
There are clearly two linearly independent Jacobi fields associated to the rotations. They are:
\[
\Phi_\tau^{R, {\tt e}_1}, \quad \Phi_\tau^{R, {\tt e}_2},
\]
and they correspond to  rotations about the coordinate axis. Note that in isothermal coordinates functions $\Phi_\tau^{R, {\tt e}_j}$, $j=1,2$ grow linearly as functions of $s$.  
\item[(3)]{\it The Jacobi field associated with the variation of the Delaunay parameter.}
We define:
\[
\Phi^D_\tau=-\partial_\tau X_\tau\cdot N_\tau.
\]
This Jacobi field is somewhat harder to write explicitly however it can be  shown that the function $\Phi^D_\tau(s)$ is linearly growing.
\end{itemize}
In summary, the Jacobi operator ${\mathcal J}_{D_\tau}$ has at least $6$ explicit Jacobi fields which are either linearly growing or bounded. We   know that these are {\it all} Jacobi fields with temperate  growth. 

\subsection{The Fermi  coordinates near  Delaunay unduloids}\label{sec fermi}

\setcounter{equation}{0} 

Let $D_\tau$ be a Delaunay unduloid as above and,  let  $H_{D_\tau}$ denote its mean curvature.  By $N_\tau$ we will denote its inner unit normal. We will assume that there exists  a tubular neighborhood $\mathcal N_\delta$   of $D_\tau$ of width  $2\delta$ in which we can introduce local system of coordinates (Fermi coordinates) $(y, t)\in D_\tau \times (-\delta, \delta)$  setting:
\[
{\tt x}\longmapsto (y, t), \quad \mbox{where}\ {\tt x}=y+t N_\tau (y).
\]
We suppose that this map,  which we denote by $Y$,  is in fact a diffeomorphism from $\mathcal N_\delta$ to $D_\tau \times (-\delta, \delta)$ whenever $\delta$ is taken sufficiently small. In the sequel we will use the inverse of this  map
\[
\begin{aligned}
Y^{-1}\colon D_\tau \times(-\delta, \delta)&\longrightarrow \mathcal N_\delta\\
(y, t)&\longmapsto {\tt x}.
\end{aligned}
\]
Given a function $w\colon\mathcal N_\delta\to \R^n$ we define its pullback $Y^*w$ to $D_\tau \times(-\delta, \delta)$ by the diffeomorphism $Y$ as:
\[
Y^* w({y,t})=w\circ Y^{-1}(y,t).
\]
For technical reasons we will chose later the size of the tubular neighbourhood $\delta$ depending on $\ve$ but for now on we just take $\delta$ small. 

{
We will now derive formulas expressing the Laplace operator   $\Delta$  in $\R^3$  in terms of  the Fermi coordinates $(y, t)\in D_\tau \times (-\delta, \delta)$. We define for each $t\in(-\delta, \delta)$
\begin{align*}
D_{\tau, t}=\{x\in \mathcal N_\delta\mid\hbox{dist}\,(D_\tau, x)=t\}.
\end{align*}
In other words $D_{\tau, t}$ is the surface obtained from $D_\tau$ by translation in the direction of the normal by $t$. Then the well known formula gives:
\begin{align}
\Delta = \Delta_{D_{\tau, t}}+\partial_{tt}-H_{D_{\tau, t}}\partial_{t},
\label{laplace}
\end{align}
where $H_{D_{\tau, t}}$ denotes the mean curvature of ${D_{\tau, t}}$.  We need to expand these operators in terms of the variable $t$. By $g$ and $g_{{t}}$, respectively, we will denote the metric  on $D_\tau$, $D_{\tau, t}$ (induced from $\R^{3}$). Let us fix a point on $D_\tau$ and   some local parametrisation  $X(u)$, $u\in \mathcal U\subset \R^2$ of $D_\tau$ in a neighbourhood of this point ($X$ could be the isothermal coordinates but any parametrization will do).  In terms of these local coordinates  we get  the following relation:
\begin{align}
\label{metric 1}
g_{z,ij}=g_{ij}+t a_{ij}+ t^2 b_{ij},
\end{align}
where
\begin{align}
\label{metric 2}
\begin{aligned}
g_{ij}&=(\partial_{u_j} X\cdot \partial_{u_i} X), \quad a_{ij}=(\partial_{u_j} X\cdot\partial_{u_i} N_\tau)+(\partial_{u_i}X\cdot\partial_{u_j} N_\tau),  \\
b_{ij}&=(\partial_{u_i} N\cdot\partial_{u_j} N_\tau).
\end{aligned}
\end{align}
Then, for the matrix  $g^{-1}_{{t}}=(g^{ij}_{{t}})_{i,j=1,\dots, 2}$ we get, provided that $|t|$ is sufficiently small:
\begin{align}
\label{ginverse}
g^{-1}_{{t}}=g^{-1}+t M_1+t^2 M_2,
\end{align}
where 
\begin{align*}
M_1=M_1(u), \quad  M_2=M_2(u,t),
\end{align*} 
are smooth matrix  functions. The expression for the Laplace-Beltrami operator on $D_\tau$ in local coordinates is:
\begin{align*}
\Delta_{D_\tau}&=\frac{1}{\sqrt {\mathrm {det}\,(g)}}\partial_{u_j}\left(\sqrt {\mathrm {det}\,(g)}g^{ij}\partial_{u_i}\right)\\
&=g^{ij}\partial_{u_iu_j}+\frac{1}{\sqrt {\mathrm {det}\,(g)}}\partial_{u_j}\left(\sqrt {\mathrm {det}\,(g)}g^{ij}\right)\partial_{u_i}\\
&=g^{ij}\partial_{u_iu_j}-g^{k\ell}\Gamma_{k\ell}^i\partial_{u_i},
\end{align*}
where $\Gamma_{k\ell}^i$ are the Christoffell symbols. A  similar formula holds for $\Delta_{D_{\tau, t}}$. Using this we can write:
\begin{align*}
\Delta_{D_{\tau, t}}=\Delta_{D_\tau}+{c}_{ij}\partial_{u_iu_j}+{d}_{i}\partial_{u_i}.
\end{align*}
where 
\begin{equation}
\label{expr ab local}
\begin{aligned}
{c}_{ij}&=g^{ij}_{{t}}-g^{ij},\\
{d}_{i}&=g^{k\ell}_{{t}}\left(\Gamma_{t, k\ell}^i-\Gamma_{k\ell}^i\right)+\Gamma_{k\ell}^i\left(g^{k\ell}_{{t}}-g^{k\ell}\right).
\end{aligned}
\end{equation}
Expressions in local coordinates for ${c}_{ij}$, ${d}_{i}$ can be further  derived using the above expansions, however their exact form is not crucial here. The point is that these functions are  small in terms of $|t|$:
\begin{equation}
\label{ab est}
|{c}_{ij}(u,t)|+|{d}_i(u,t)|\leq C|t|.
\end{equation}
With a choice of local coordinates on $D_\tau$ the constant in the above estimate does not depend on the point on $D_\tau$.  

Next, we will expand the mean curvature $H_{D_{\tau,z}}$. To this end  by ${\Bbbk}_{j}$, $j=1, 2$ we will denote the principal curvatures of $D_\tau$. Then we have
\begin{align}
\label{expansion of H}
\begin{aligned}
H_{D_{\tau, t}}&=\sum_{j=1}^2\frac{\Bbbk_{j}}{1-t\Bbbk_{j}}\\
&=\sum_{j=1}^2\Bbbk_{j}+t\sum_{j=1}^2 \Bbbk_{j}^2+ \mathbb{Q},\\
&=H_{D_\tau}+t |A_{D_\tau}|^2+\mathbb Q,
\end{aligned}
\end{align}
where
\begin{align}
\label{expr q}
\mathbb Q(y,t) =t^2\sum_{j=1}^2 \Bbbk_{j}^3+t^3\sum_{j=1}^2 \Bbbk_{\ve, j}^4+\dots,
\end{align}
and $|A_{D_\tau}|$ is the norm of the second fundamental form on $D_\tau$.  Summarizing all this using   (\ref{laplace}) we can express the Laplace operator in Fermi coordinates as follows
\begin{equation}
\label{expr fermi}
\Delta = \partial_{tt} +\Delta_{D_\tau} -\left(H_{D_\tau}+t |A_{D_\tau}|^2+\mathbb Q\right)\partial_t+\mathbb A,
\end{equation}
where $\mathbb A$  is a differential operator whose coefficients are given in  (\ref{expr ab local}) and satisfy (\ref{ab est}).

Next we introduce  stretched  Fermi coordinates 
\[
{\tt t}=\frac{t}{\ep}, \quad {\tt y}=y.
\]
As before we have a diffeomorphism $Y_{\ve}$ and its inverse  $Y^{-1}_{\ve}\colon D_\tau \times(-\frac{\delta}{\ve}, \frac{\delta}{\ve})\to \mathcal N_{\delta}$, and for any function $w\colon\mathcal N_{\delta}\to \R^k$ we define its pullback by $Y_{\ve}$ by:
\[
Y^*_{\ve} w({\tt y}, {\tt t})=w\circ Y^{-1}_{\ve}({\tt y}, {\tt t}).
\] 
Taking into account formula (\ref{expr fermi}) we get
\begin{equation}
\label{fermi 2}
\Delta=\ve^{-2}\partial_{\st \st}-\ve^{-1}(H_{D_\tau}+\ve{\tt t}|A_{D_\tau}|^2+\mathbb Q_\ve)\partial_{\tt t}+\Delta_{D_\tau}+\mathbb A_\ve.
\end{equation}
where
\[
\mathbb Q_\ve(\sy, \st)=\mathbb Q(\sy, \ve \st), \qquad \mathbb A_\ve(\sy, \st)=\mathbb A(\sy, \ve\st).
\]


\subsection{Two ended Delaunay solutions of the Cahn-Hilliard equation} \label{sec two end}
Locally near the surface $D_\tau$ the function $w_\tau$, which is the solution of the Cahn-Hilliard equation described   in Theorem \ref{teorema 1} should, to main order, depend on the stretched Fermi variable $\st$ only. 
To find this first approximation of $w_\tau$ we use (\ref{fermi 2}) where we ignore  terms of order  $o(\ve)$.   Arguing formally we are lead to solving the following problem:
 \begin{equation}
\label{eq U}
\begin{aligned}
U''-\ve H_{D_\tau} U'+f(U)&=\ve\ell_\ve, \quad \mbox{in}\ \R,\\
f\left(U(\pm \infty)\right)&=\ve\ell_\ve,
\end{aligned}
\end{equation}
where we also have to determine the  Lagrange multiplier $\ell_\ve$. 
This function can be found by a straightforward perturbation argument assuming  $U=\varTheta+\mathcal O(\ve)$ where $\varTheta$ is the unique odd, monotonically increasing solution of the Allen-Cahn  equation (equivalent to setting $\ve=0$ in (\ref{eq U})), see \cite{ch_arxiv} for details. 
We have $U(\pm\infty)=\pm 1+\sigma_\ve^\pm$, where 
\begin{equation}
\label{def sigma}
f(\pm 1+\sigma_\ve^\pm)=\ve\ell_\ve.
\end{equation}
Also,
\[
\ell_\ve=\ell_0+\mathcal O(\ve), \quad \ell_0=-\frac{1}{2}H_{D_\tau}\int_\R \varTheta'(s)^2\,ds.
\]

Now we will briefly summarise some of the tools and results in \cite{ch_arxiv}.  From the proof of  Theorem \ref{teorema 1} we can describe the local   behaviour of $w_\tau$ in more details. To this end it is useful  to express $w_\tau$ near $D_\tau$ in  the local stretched Fermi coordinates $(\sy, \st)$ introduced above.    
We define weighted H\"older norms on $D_\tau \times \R$ by:
\begin{equation}\label{holder 2}
\begin{aligned}
\|u\|_{\mathcal C^{0,\alpha}_\mu(D_\tau \times \R)}&=\sup_{\st\in \R}\,(\cosh \st)^\mu
\|u\|_{\mathcal C^{0,\alpha}(D_{\tau}\times (\st -1, \st+1))},\\
{\|u\|_{\mathcal C^{1,\alpha}_\mu(D_{\tau}\times
\R)}}&={\|u\|_{\mathcal
C^{0,\alpha}_\mu(D_{\tau}\times \R)}+\|\nabla_{D_{\tau}\times \R}u\|_{\mathcal C^{0,\alpha}_\mu(D_{\tau}\times
\R)}},\\
\|u\|_{\mathcal C^{2,\alpha}_\mu(D_{\tau}\times \R)}&=\|u\|_{\mathcal
C^{0,\alpha}_\mu(D_{\tau}\times \R)}+\|\nabla_{D_{\tau}\times \R}u\|_{\mathcal C^{0,\alpha}_\mu(D_{\tau}\times
\R)}+\|{\nabla}^2_{D_{\tau}\times \R}u\|_{\mathcal
C^{0,\alpha}_\mu(D_{\tau}\times \R)}.
\end{aligned}
\end{equation}
With these definitions  there exists $\mu_0>0$ and $\alpha_0>0$ such that  for $0<\mu<\mu_0$ and $0<\alpha<\alpha_0$ it holds:
\begin{equation}
\label{asymp wtau}
Y^*_{\ve}w_\tau (\sy, \st)=U(\st)+\mathcal O_{\mathcal C^{2,\alpha}_{\mu} (D_\tau\times \R)}(\ve^{2-\alpha}).
\end{equation}
Above the symbol $\mathcal O_{\mathcal C^{2,\alpha}_{\mu} (D_\tau\times \R)}(\ve^{2-\alpha})$ denotes functions whose $\mathcal C^{2,\alpha}_{\mu} (D_\tau\times \R)$ norm is bounded by a constant times $\ve^{2-\alpha}$.  This formula is valid in a tubular neighbourhood $\mathcal N_{\delta(\ve)}$ of $D_\tau$, where $\delta(\ve)=\mathcal O(\ve^{{2}/{3}})$. In local variables this means $|\st|\leq C\ve^{-1/3}$. 
Outside of this neighbourhood we have
\begin{equation}
\label{asymp wtau 2}
w_\tau=\pm 1+\sigma_\ve^\pm+\psi, \qquad \mbox{where}\ \|\psi\|_{\mathcal C^2(\R^3)}\leq C e^{\,-c/\ve^{1/3}}. 
\end{equation}
In fact more is true.
We claim that  $w_\tau$  converges  exponentially to constants $\pm 1+\sigma_\ve$ away from $D_\tau$. More precisely, if $D_\tau$ is given as surface of revolution of the curve $x_1=\rho_\tau(z)$ then 
\begin{equation}
\label{exp decay w}
|w_\tau(r, z)+1-\sigma_\ve^-|\leq C\exp\big[-\frac{c}{\ve}(r-\rho_\tau(z))\big], \quad r>\rho_\tau(z),   \quad r=\sqrt{x_1^2+x_2^2},
\end{equation}
with similar estimate when $r<\rho_\tau(z)$. To prove this we note that (\ref{exp decay w}) is  valid in a tubular neighbourhood of $D_\tau$ by (\ref{asymp wtau}) and the fact that $\st$ and 
$\frac{(r-\rho_\tau(z))}{\ve}$ are comparable in this neighbourhood. Far form $D_\tau$ we use the fact that $w_\tau=\pm 1+\sigma_\ve^\pm+\psi$, where $\psi$ is an exponentially small in $\ve$ function (see (\ref{asymp wtau 2})) and a comparison argument. 
These estimates can be made more precise as far as the rate of exponential decay but we will not need such a precision here. 

One property that we will need in the sequel is differentiability of $w_\tau$ with respect to $\tau\in (0,1)$. Although this is not explicitly stated in \cite{ch_arxiv} this property also follows from the proof of Theorem  \ref{teorema 1} by a rather standard argument using the version of the Banach fixed point theorem in \cite{MR660633}. We will omit the details pointing out only that the ansatz $U(\st)$ is a differentiable function of $\tau$ since the Fermi coordinate is a smooth function of $\tau$. For future reference we note that  we have on $D_\tau$ (i.e. $\st=0$)
\[
\partial_\tau \st =\ve^{-1} \partial_\tau X_\tau\cdot N_\tau=-\ve^{-1}\Phi^D_\tau,
\] 
where $\Phi^D_\tau$ is the Jacobi field on $D_\tau$ associated with the change of the Delaunay parameter. 

\subsection{The linearized operator near $D_\tau$}\label{sec 222}
Our main objective in the next  section will be  to study the linearized operator of the Delaunay solution 
\begin{equation}
L_{w_\tau}=\ve\Delta+\frac{1}{\ve} f'(w_\tau),
\label{def lwtau}
\end{equation}
as an operator defined for functions on $\R^3$ and here we will introduce some basic observations and notations needed later.

Using (\ref{fermi 2}) we find  expression of  $L_{w_\tau}$ in stretched  Fermi coordinates in   $\mathcal N_{\delta}$:
\begin{equation}\label{fermi tilde L}
\begin{aligned}
Y^*_{\ve} L_{w_\tau}u&=\ve^{-1}\left[\partial_{\st \st} u-\ve (H_{D_\tau}+\ve{\tt t}|A_{D_\tau}|^2+\mathbb Q_\ve)\partial_{\tt t}u +f'(w_\tau)u\right]+\ve \Delta_{D_\tau}u +\ve \mathbb A_\ve u,
\end{aligned}
\end{equation}
where with some abuse of notation we write $u$ and $w_\tau$ instead of $Y_\ve^* u$ and $Y_\ve^* w_\tau$ (we will consistently abuse notation this way whenever it is unambiguous).
One technical problem we will have to face in this paper is the fact that while the operator $L_{w_\tau}$ is defined in $\R^3$ its expression in local coordinates $Y^*_{\ve}L_{w_\tau}$ makes sense only in $\mathcal N_\delta$ and not as we would like in $D_\tau\times \R$. There are possibly many ways to extend $Y^*_{\ve}L_{w_\tau}$ and  we will  chose one of them for the rest of the paper.   
Let $\chi(s)$ be a smooth nonnegative cut-off function equal to $1$ for $|s|\leq 1$ and equal to $0$ for $|s|>2$.  We set
\begin{equation}
\chi_{\ve/\delta}(\st)=\chi\left(\frac{\ve \st}{\delta}\right).
\label{def cutoff}
\end{equation}
We need to extend the function $Y^*_{\ve}w_\tau$ in such a way that it is defined outside of $\mathcal N_\delta$. To this end we set
\[
{\tt w}_\tau=\chi_{\ve/\delta}(\st) Y^*_{\ve} w_\tau+\left(1-\chi_{\ve/\delta}(\st)\right) U(\st). 
\] 
Next we define the extension of the operator $Y^*_{\ve}L_{w_\tau}$ by
\begin{equation}
\label{ext l wtau}
 \begin{aligned}
\mathbb L_{w_\tau}u&={\ve^{-1}}\left[\partial_{\st\st} u    -\ve \left(H_{D_{\tau}}+\ve\st \chi_{\ve/\delta}(\st)|A_{D_\tau}|^2+\chi_{\ve/\delta}(\st)\mathbb Q_\ve\right)\partial_\st u+f'({\tt w}_\tau)u \right] +\ve\Delta_{D_\tau} u +\ve \chi_{\ve/\delta}(\st)\mathbb A_\ve u.
\end{aligned}
\end{equation}
As we will see  $\mathbb L_{w_\tau}$   resembles  the operator
\[
\mathcal L u=\frac{1}{\ve}\left[\partial_{\st \st}u+f'(\varTheta) u\right] +\ve\left[\Delta_{D_\tau} u +|A_{D_\tau}|^2u \right] 
\]
whose kernel  is fairly easy to determine by separation of variables. Indeed, taking $u=\varTheta'(\st)\psi(\sy)$ we get 
\begin{equation}
\label{separate}
\mathcal L (\varTheta' \psi)=\ve \varTheta' \left[\Delta_{D_\tau}\psi +|A_{D_\tau}|^2\psi \right] 
\end{equation}
and therefore the Jacobi fields of $D_\tau$ determine the Jacobi fields of $\mathcal L$.   Let us explain in what sense $L_{w_\tau}$  and $\mathcal L$ are similar.  To do this we will use the operator $\mathbb L_{w_\tau}$ (our theory of the operator $L_{w_\tau}$ is based on exploiting this link). First we need a function which will play a role of  $\varTheta'(\st)$. Since our proof is based on a perturbation argument there is no unique way to define such a function   but a natural candidate  seems to be $\partial_{\st} Y^*_{\tau, h_0} w_\tau$.  An  important  observation to make  is that $\partial_{\st}$ and $\Delta$  do not commute so   we do not have $\mathbb L_{w_\tau}\partial_\st w_\tau=0$ and as we will see below the  commutator $[\Delta, \partial_\st]$  gives rise to the term $|A_{D_\tau}|^2$ in $\mathcal L$. Since  $\partial_\st Y^*_{\tau, h_0} w_\tau$ is defined only in $\mathcal N_\delta$ we define the  extension of this function to  $D_\tau\times \R$ by 
\begin{equation}
\label{def vv}
{\tt V}=\chi_{\ve/\delta}(\st)\partial_\st Y^*_{\ve} w_\tau+\left(1-\chi_{\ve/\delta}(\st)\right) U'(\st),
\end{equation}
where $U$ is the solution of (\ref{eq U}). Note that $V$ depends on $\st\in \R$ and $\sy\in D_\tau$ but using (\ref{asymp wtau}) we get 
\begin{equation}
{\tt V}(\sy, \st)= U'(\st)+\mathcal O_{C^{1,\alpha}_\mu(D_\tau\times \R)}(\ve^{2-\alpha}),
\label{vasymp}
\end{equation}  
globally on $D_\tau\times \R$, which means that the dependence on $\sy$ is mild.  
Next we  calculate 
\begin{equation}
\label{expr v}
\begin{aligned}
\mathbb L_{w_\tau} {\tt V}&=\chi_{\ve/\delta} \mathbb L_{w_\tau} \partial_{\st} Y^*_{\tau, h_0}w_\tau+\left(1-\chi_{\ve/\delta}\right) \mathbb L_{w_\tau}  U'+\left[\mathbb L_{w_\tau}, \chi_{\ve/\delta}\right] \partial_{\st}Y^*_{\tau, h_0}w_\tau+\left[\mathbb L_{w_\tau}, 1-\chi_{\ve/\delta}\right] U'.
\end{aligned}
\end{equation}
The first term above is the most complicated. For brevity let us denote $v_\tau=\partial_{\st}Y^*_{\tau, h_0}w_\tau$. With this notation differentiating the equation satisfied by $w_\tau$ in $\mathcal N_\delta$ with respect to $\st$  we have
\[
\mathbb L_{w_\tau} v_\tau - \ve |A_{D_\tau}|^2 v_\tau = \ve \left[\mathbb Q_\ve, \partial_{\st}\right] v_\tau+\ve \left[\mathbb A_\ve, \partial_\st\right] Y_\ve^* w_\tau,
\]
where  $\left[A, B\right]=AB-BA$.
By definition of $\mathbb Q_\ve$ we see that 
\[
\left[\mathbb Q_\ve, \partial_{\st}\right] v_\tau=-\ve^2 \left(2\st   \sum_{j=1}^2 \Bbbk_{j}^3+3\ve \st^2\sum_{j=1}^2 \Bbbk_{\ve, j}^4+\dots \right) v_\tau=\mathcal O_{C^{1, \alpha}_\mu(D_\tau\times \R)} (\ve^2)
\]
The differential operator $\mathbb A_\ve$ contains derivatives in $\sy\in D_\tau$ only while $Y_\ve^* w_\tau$ is, up to order $\mathcal O_{C^{2, \alpha}_\mu(D_\tau\times \R)} (\ve^{2-\alpha})$, a function of $\st$. This gives
\[
\ve \left[\mathbb A_\ve, \partial_\st\right] Y_\ve^* w_\tau=\ve \mathbb A_\ve\partial_\st Y_\ve^* w_\tau-\ve\partial_t \mathbb A(\sy, \ve \st)Y_\ve^* w_\tau=\mathcal O_{C^{0, \alpha}_\mu(D_\tau\times \R)} (\ve^{3-\alpha}).
\]
It follows that in $\mathcal N_\delta$ we get
\[
\mathbb L_{w_\tau} v_\tau - \ve |A_{D_\tau}|^2 v_\tau=\mathcal O_{C^{0, \alpha}_\mu(D_\tau\times \R)} (\ve^{3-\alpha})
\]
Considering other terms in (\ref{expr v}) from  the fact that $\chi_{\ve/\delta}=\chi_{\ve/\delta}(\st)$ and (\ref{asymp wtau}) we get
\[
\left[\mathbb L_{w_\tau}, \chi_{\ve/\delta}\right] \partial_{\st}Y^*_{\tau, h_0}w_\tau=\mathcal O_{C^{0, \alpha}_\mu(D_\tau\times \R)} (\ve^{3-\alpha}).
\]
Similar estimates hold for terms involving $U'(\st)$. In summary we get
\begin{equation}
\label{expr v2}
\mathbb L_{w_\tau} {\tt V}-\ve |A_{D_\tau}|^2 {\tt V}=\mathcal O_{C^{0, \alpha}_\mu(D_\tau\times \R)} (\ve^{3-\alpha}).
\end{equation}

Now let $\psi\in C^{2, \alpha}(D_\tau)$ be fixed. Using (\ref{expr v2}) we get
\begin{equation}
\label{expr v3}
\begin{aligned}
\mathbb L_{w_\tau} (\psi{\tt V})&=\psi \mathbb L_{w_\tau} {\tt V}+\ve {\tt V} \left(\Delta_{D_\tau} +\chi_{\ve/\delta}\mathbb A_\ve\right) \psi+\ve \left[\Delta_{D_\tau} +\chi_{\ve/\delta}\mathbb A_\ve, {\tt V}\right]\psi\\
&= \psi\left( \mathbb L_{w_\tau} {\tt V}-\ve|A_{D_\tau}|^2{\tt V}\right)+\ve {\tt V}\left(\Delta_{D_\tau} +|A_{D_\tau}|^2+\chi_{\ve/\delta}\mathbb A_\ve\right) \psi
+\ve \left[\Delta_{D_\tau} +\chi_{\ve/\delta}\mathbb A_\ve, {\tt V}\right]\psi-\ve\psi(\Delta+\chi_{\ve/\delta}){\tt V}\\
&=\ve{\tt V} \left(\mathcal J_{D_\tau}+\chi_{\ve/\delta}\mathbb A_\ve\right)\psi+ \ve \left[\Delta_{D_\tau} +\chi_{\ve/\delta}\mathbb A_\ve, {\tt V}\right]\psi+\mathcal O_{C^{0, \alpha}_\mu(D_\tau\times \R)} (\ve^{3-\alpha})\psi,
\end{aligned}
\end{equation}
where $\mathcal J_{D_\tau}$ is the Jacobi operator on $D_\tau$. For future reference we note that 
\begin{equation}
\label{expr v4}
\left\|\left[\Delta_{D_\tau} +\chi_{\ve/\delta}\mathbb A_\ve, {\tt V}\right]\psi\right\|_{C^{0,\alpha}_\mu(D_\tau\times \R)}\leq C\ve^{2-\alpha}\|\psi\|_{C^{1,\alpha}(D_\tau)}.
\end{equation}
Observe that formula (\ref{expr v3}) is quite similar to (\ref{separate}) and in particular it is clear that if $\mathbb L_{w_\tau} (\psi{\tt V})\approx 0$ then  $\psi$ should be  a Jacobi field on $D_\tau$, and as a consequence  we should get an approximate Jacobi field of $w_\tau$.
Indeed we can easily find  describe explicit   Jacobi fields of the two ended Delaunay solution $w_\tau$ which are approximately of the form $\psi {\tt V}$.  Let   ${\tt h}=\sum_{i=1}^3 h_i {\tt e}_i$ be a vector,   $\mathcal R_{\vartheta}({\tt x})=\mathcal R_{\vartheta_1, \vartheta_2}({\tt x})$ be a rotation in $\R^3$, where $\vartheta_i$ is the angle of the rotation about the $x_i$ axis $i=1,2$,   and $\eta$ be a number  such that  $|\eta|$ is small. Then the function
\[
\Phi_{{\tt h}, \vartheta, \eta}(w_\tau)=(w_{\tau+\eta}\circ \mathcal R_\vartheta)({\tt x}+{\tt h}),
\]
is also a solution of the Cahn-Hilliard equation (\ref{CH}). In particular, taking derivatives of $\Phi_{{\tt h}, \vartheta, \eta}(w_\tau)$ with respect to the parameters we get 
\[
\begin{aligned}
L_{w_\tau} \partial_{h_i} \Phi_{{\tt h}, \vartheta, \eta}(w_\tau)\mid_{{\tt h}, \vartheta, \eta=0}&=0, \quad i=1,2,3,\\
L_{w_\tau} \partial_{\vartheta_i} \Phi_{{\tt h}, \vartheta, \eta}(w_\tau)\mid_{{\tt h}, \vartheta, \eta=0}&=0, \quad i=1,2,\\
L_{w_\tau} \partial_{\eta} \Phi_{{\tt h}, \vartheta, \eta}(w_\tau)\mid_{{\tt h}, \vartheta, \eta=0}&=0, 
\end{aligned}
\]
and hence the $6$ dimensional linear space 
\begin{equation}\label{def kernel}
\mathcal I_{w_\tau}=
\mathrm {span}\, \{\partial_{h_i} \Phi_{{\tt h}, \vartheta, \eta}(w_\tau)\mid_{{\tt h}, \vartheta, \eta=0}, \quad \partial_{\vartheta_i} \Phi_{{\tt h}, \vartheta, \eta}(w_\tau)\mid_{{\tt h}, \vartheta, \eta=0}, \quad \partial_{\eta} \Phi_{{\tt h}, \vartheta, \eta}(w_\tau)\mid_{{\tt h}, \vartheta, \eta=0}\}.
\end{equation}
These are the geometric  Jacobi fields  of $L_{w_\tau}$ introduced in already in the inrtoduction. 
For future use  we state the following lemma:
\begin{lemma}\label{lemma asymptot kernel}
With the above notations the following formulas hold in a tubular neighbourhood $\mathcal N_\delta(\ve)$, $\delta(\ve)=\mathcal O(\ve^{\frac{2}{3}})$:
\begin{equation}
\label{app ker 1}
\begin{aligned}
Y^*_{\ve}\partial_{h_i}\Phi_{{\tt h}, \vartheta, \eta}(w_\tau)\mid_{{\tt h}, \vartheta, \eta=0}&=\ve^{-1} \Phi_\tau^{T, {\tt e}_i}{\tt V}+\mathcal O_{C^{1,\alpha}_{\mu} (D_\tau\times \R)}(1),\\
Y^*_{\ve}\partial_{\vartheta_i} \Phi_{{\tt h}, \vartheta, \eta}(w_\tau)\mid_{{\tt h}, \vartheta, \eta=0}&=\ve^{-1}\Phi_\tau^{R, {\tt e}_i}{\tt V}+\mathcal O_{C^{1,\alpha}_{\mu} (D_\tau\times \R)}(1),\\
Y^*_{\ve}\partial_{\eta} \Phi_{{\tt h}, \vartheta, \eta}(w_\tau)\mid_{{\tt h}, \vartheta, \eta=0}&=\ve^{-1}\Phi_\tau^{D}{\tt V}+\mathcal O_{C^{1,\alpha}_{\mu} (D_\tau\times \R)}(1).\\
\end{aligned}
\end{equation}
\end{lemma}
\begin{proof}
We recall that by  (\ref{asymp wtau}) in $\mathcal N_{\delta(\ve)}$ we have 
\begin{equation}
\label{app ker 2}
Y^*_{\ve}w_\tau (\sy, \st)=U(\st)+\mathcal O_{\mathcal C^{2,\alpha}_{\mu} (D_\tau\times \R)}(\ve^{2-\alpha}).
\end{equation}
In $\mathcal N_{\delta(\ve)}$ we can write explicitly using the isothermal coordinates on $D_\tau$:
\begin{equation}
{\tt x}= X_\tau (s,\theta)+\ve \st  N_\tau(s, \theta).
\label{app ker 3}
\end{equation}
Now, fix  a unit vector ${\tt e}\in \R^3$ and denote ${\tt x}_h={\tt x}+h{\tt e}$. Taking derivative in $h$ of (\ref{app ker 3}) and evaluating at $h=0$ we get:
\[
{\tt e}=\ve \partial_{\tt e} \st N_\tau+\partial_{\tt e} s[\partial_s X_\tau+\ve \st \partial_s N_\tau]+\partial_{\tt e} \theta[\partial_\theta X_\tau+\ve \st \partial_\theta N_\tau].
\]
Taking the scalar product with $N_\tau$, $\partial_s X_\tau$ and $\partial_\theta X_\tau$ we find expression for $\partial_{\tt e}\st $, $\partial_{\tt e} s$, $\partial_{\tt e} \theta$. Note in particular that $\partial_{{\tt e}_i}\st=\ve^{-1}{\tt e}_i \cdot N_\tau=\ve^{-1} \Phi_\tau^{T, {\tt e}_i}$. Then, taking derivatives $\partial_{\tt e_i}$  of  (\ref{app ker 2})  we get the first formula in (\ref{app ker 1}).  We follow a similar argument to show the two remaining identities. 
\end{proof}

\section{Proof of Theorem \ref{main theorem}}\label{section 3}

\subsection{A functional analytic setting for $L_{w_\tau}$}

We will introduce suitable weighted Sobolev norms to study the invertibility theory for $L_{w_\tau}$. Let $\mathrm{dist}\, ({\tt x}, D_\tau)$ denote the signed distance function, where we chose the orientation of $D_\tau$ in such a way that the sign of $\mathrm{dist}\, ({\tt x}, D_\tau)$ agrees with that of $\rho_\tau(z)-r$. We have globally
\[
|\mathrm{dist}\, ({\tt x}, D_\tau)|\leq |\rho_\tau(z)-r|,
\]
and the two quantities are comparable near $D_\tau$. Recall that above we have denoted ${\st}=\frac{1}{\ve}\mathrm{dist}\, ({\tt x}, D_\tau)$ as long as $|\mathrm{dist}\, ({\tt x}, D_\tau)|\leq {\delta}$.

We will  define the weighted Sobolev norms  we will use in the sequel. First, let us consider Sobolev spaces $L^2(D_\tau\times \R)$ and $H^s(D_\tau\times \R)$. Since  functions in these spaces can be expressed in terms of the isothermal coordinates $(s,\theta)$ and the Fermi coordinate $\st$  we  define
\[
L^2_{a,\gamma}(D_\tau\times \R)=\cosh^{-a}(s)\cosh^{-\gamma}(\st)L^2(D_\tau\times \R), \quad H^s_{a,\gamma}(D_\tau\times \R)=\cosh^{-a}(s)\cosh^{-\gamma}(\st)H^s(D_\tau\times \R).
\]

Second, let us consider the subspace $H^\ell(\R^3)_+$ (respectively $H^\ell(\R^3)_-$) of $H^\ell(\R^3)$  which consists  functions supported in the set $\{z\geq -1\}$ (respectively in $\{z\leq 1\}$). We define  weighted Sobolev norms in these subspaces  as follows 
\begin{equation}\label{metal urbain}
\begin{aligned}
\|u\|_{H^\ell_{a, \gamma}(\R^3)_+}&=\sum_{|\alpha|=0}^\ell  \|e^{\, a z} e^{\,\gamma\big(\frac{r-\rho_\tau(z)}{\ve}\big)}D^\alpha u\|_{L^2(\R^3)_+}, \\
\|u\|_{H^\ell_{a, \gamma}(\R^3)_-}&=\sum_{|\alpha|=0}^\ell  \|e^{\, -a z} e^{\,\gamma \big(\frac{r-\rho_\tau(z)}{\ve}\big)}D^\alpha u\|_{L^2(\R^3)_-},
\end{aligned}
\end{equation}
where $\alpha=(\alpha_1, \alpha_2, \alpha_3)$ is a multi index and derivatives are taken with respect to $x_j$, $j=1,2,$ and $z$ (which we will identify with $x_3$ when convenient).  Note that $\gamma$ measures the rate of decay or growth of the solution in the transversal direction  to $D_\tau$ and $a$ measures the rate of decay or growth along the  axis of $D_\tau$ in the positive (respectively negative) direction.  Finally, we define 
\[
H^\ell_{a, \gamma}(\R^3)=H^\ell_{a, \gamma}(\R^3)_+\oplus H^\ell_{a, \gamma}(\R^3)_-.
\]
With these definitions when $\gamma>0$, $a>0$ our spaces consist of exponentially decaying functions, in the opposite case they are exponentially increasing. Combinations of signs for $\gamma$ and $a$ are of course allowed. 

The norms $H^s_{a,\gamma}(D_\tau\times I_{\delta/\ve})$, where $I_{\delta/\ve}=(-\delta/\ve, \delta/\ve)$,  and $H^s_{a, \gamma}(\R^3\cap \{|\mathrm{dist}\, ({\tt x}, D_\tau)| <\delta\})$ are equivalent in the following sense:
\begin{equation}
\label{nick cave 1}
\begin{aligned}
\|\phi\|_{L^2_{a,\gamma}(\R^3\cap \{|\mathrm{dist}\, ({\tt x}, D_\tau)| <\delta\})}&\leq C\ve^{1/2}\|Y^*_{\ve} \phi\|_{L^2_{a^*, \gamma^*}(D_\tau\times I_{\delta/\ve})},\\
\|\phi\|_{L^2_{a,\gamma}(\R^3\cap \{|\mathrm{dist}\, ({\tt x}, D_\tau)| <\delta\})}&\geq C\ve^{1/2}\|Y^*_{\ve} \phi\|_{L^2_{a_*, \gamma_*}(D_\tau\times I_{\delta/\ve})},
\end{aligned}
\end{equation}
where in general constants $a, \gamma$, $a^*, \gamma^*$ and $a_*, \gamma_*$ are different. In addition, relating the norms of gradients and second derivatives we expect to loose powers of $\ve$. For instance
\begin{equation}
\label{nick cave 2}
\begin{aligned}
\|\nabla\phi\|_{L^2_{a,\gamma}(\R^3\cap \{|\mathrm{dist}\, ({\tt x}, D_\tau)| <\delta\})}&\leq C\ve^{-1/2}\|\nabla Y^*_{\ve} \phi\|_{L^2_{a^*, \gamma^*}(D_\tau\times I_{\delta/\ve})},
\\
\|\nabla \phi\|_{L^2_{a,\gamma}(\R^3\cap \{|\mathrm{dist}\, ({\tt x}, D_\tau)| <\delta\})}&\geq C\ve^{1/2}\|\nabla Y^*_{\ve} \phi\|_{L^2_{a_*, \gamma_*}(D_\tau\times I_{\delta/\ve})}.
\end{aligned}
\end{equation}
Similar estimates hold for the second derivatives. We will use these estimates  later on.

\subsection{The Fourier-Laplace transform of $L_{w_\tau}$}

We will consider the  linear operator $L_{w_\tau}$ acting on the space $L^2_{a,\gamma}{(\R^3)}$ with dense domain $D(L_{w_\tau})=H^2_{a,\gamma}{(\R^3)}$ defined by
\[
\begin{aligned}
L_{w_\tau}\colon H^2_{a,\gamma}{(\R^3)}&\longmapsto L^2_{a,\gamma}{(\R^3)},\\
u&\longmapsto L_{w_\tau} u.
\end{aligned}
\]
The important property of the operator $L_{w_\tau}$ is the fact that it is periodic in $z$.  
This will allow us to  define  the Fourier-Laplace transform of $L_{w_\tau}$ (this idea was originated  by Taubes \cite{taubes1987}, \cite{MR755237} and developed in the form that we adopt here in  \cite{MR1356375} and \cite{MR1941630}).

To begin we define the   Fourier-Laplace transform for functions on $\R$ by:
\begin{equation}\label{fl def}
\hat h (\sigma, \zeta)=\mathcal F(h)=\sum_{-\infty<k<\infty} e^{\,-i(k+\sigma)\zeta} h(\sigma+k),\quad \sigma\in [0, 1), \quad \zeta=\mu+i\nu.
\end{equation}
Observe that with this definition we have
\begin{equation}
\label{fl quasi period}
\hat h(\sigma+1,\zeta)= \sum_{-\infty<k<\infty} e^{\,-i(k+1+\sigma)\zeta} h(\sigma+k+1)= \hat h(\sigma, \zeta).
\end{equation}
Note that the definition we adopt here is slightly different from the one in \cite{MR1356375}-the two differ by a factor $e^{\,-i\sigma\zeta}$.

The Fourier-Lapalace transform can be inverted and the inverse is given by an explicit formula. To state it let  $s\in \R$ be given and denote  the  fractional part of $s$ by   $s\modulo 1$. With this notation we have
\begin{equation}
\label{fl 1}
h(s)=\mathcal  F^{-1}(\hat h)(s)=\frac{1}{2\pi} \int_{\mu=0}^{2\pi} e^{\, i s \zeta} \hat h(s\modulo 1, \zeta)\,d\mu,
\end{equation}
where we integrate  along the line $\mathrm{Im}\, \zeta=\nu$, $\zeta=\mu+i\nu$ (see \cite{taubes1987}). The Fourier-Laplace transform is well defined in the Schwarz class $\mathcal S$ and, by Cauchy's theorem, the value of the integral in the inversion formula does not depend on $\nu$, since the segment along which we integrate can be vertically shifted. However, for our purpose it is convenient to consider the class of functions which are allowed to grow exponentially at $+\infty$ (or at $-\infty$). Suppose for instance that $h$ is a continuous function,   supported in $[-1,\infty)$ and such that $|e^{\, a s} h(s)|<\infty$.    Then the series in (\ref{fl def}) is well defined as long as $\mathrm{Im}\, \zeta=\nu<a$.  
Likewise, we can define the transform on  a subspace:
\[
H^\ell_{a}(\R)_+= e^{\,-as} H^\ell(\R),
\]
of the  Sobolev space $H^\ell(\R)$ consisting of functions supported in $[-1,\infty)$, 
where $a$ is the rate of exponential decay  or growth. In a similar way we define the subspace $H^\ell_{a}(\R)_-$ of $H^\ell(\R)$ consisting of exponentially decaying or growing functions supported in $(-\infty, 1]$. As long as  $\mathrm{Im}\, \zeta=\nu\leq a$ the Fourier-Laplace transform of $h\in H^\ell_{a}(\R)_+$ is well defined. Moreover the function $h$ can be recovered from $\hat h(\cdot, \zeta)$ if the path of integration in the formula (\ref{fl 1}) is taken in the lower half plane $\mathbb H^-_{a}=\{\mathrm{Im}\, (\zeta)=\nu \leq a\}$. 
The situation is similar when instead we consider the Fourier-Laplace transform in the space of functions $H^\ell_{a}(\R)_-$, except that now the transform is defined in the upper half plane $\mathbb H_a^+=\{\mathrm {Im}\, \zeta=\nu\geq a\}$.

We observe that from  Plancherel's formula: 
\begin{equation}
\label{fl plancherel}
\int_{\mu=0}^{2\pi}\int_0^1 |\hat h(\sigma,\zeta)|^2\, d\sigma\,d\mu=\int_\R |e^{\,\nu s} h(s)|^2\, ds, \qquad \zeta=\mu+i\nu,
\end{equation}
it follows that $L^2$ norms of the Fourier-Laplace transforms are equal to the exponentially weighted $L^2$ norm of functions. This property is crucial for our purpose. 

Note that if $u(\sigma, \zeta)$ is an $L^2([0,1])$ function which  is analytic as a function of $\zeta$ with values in $L^2([0,1])$ in the lower half plane $\mathbb H^-_{a}$ then, by Cauchy's theorem, the path of integration in the inversion formula (\ref{fl 1}) can be shifted down to any path $\zeta=\mu+i\nu$, $\nu<a$. If in addition $u(\cdot,\zeta)$ is bounded by $e^{\,-\nu}$  along such paths then  the inverse transform $\mathcal F^{-1} u(s)$ is supported in $[-1, \infty)$. This explains the reason we have paid so much  attention to functions defined on a half-line. On the other hand Fourier-Laplace transforms of functions in $H^\ell_{a}(\R)_+$ have the property described above.

The Fourier-Laplace transform plays a similar role as the Fourier transform in the theory of linear PDEs with constant coefficients when the differential operator at hand is periodic with respect  to  the independent variable.  To fix attention on a concrete example  let us suppose  that $A(s)\colon L_a^2(\R)_+\to L^2_a(\R)_+$, $s\in \R$ is a family of  densely defined, linear  operators. Then it is natural to define 
\begin{equation}
\label{fl transf oper}
\hat A(\sigma, \zeta)\hat h=\widehat{A(s)h}(\sigma,\zeta).
\end{equation}
Now, let us suppose that $A$ is periodic with period $1$, i.e.  $A(s)=A(s+1)$. We have
\[
\hat A(\sigma, \zeta)\hat h=\sum_{-\infty<k<\infty} e^{\, -i(k+\sigma)\zeta} A(\sigma + k)h(\sigma+k)=e^{\,-i\zeta\sigma}A(\sigma) e^{\,i\zeta\sigma} \hat h,
\]
hence  explicitly 
\begin{equation}
\label{fl transf oper expli}
\hat A(\sigma, \zeta)=e^{\,-i\zeta\sigma}A(\sigma) e^{\,i\zeta\sigma}.
\end{equation}
With our definition of the Fourier-Laplace transform we have $\hat h(\sigma)=\hat h(\sigma +1)$ and also $\hat A(\sigma, \zeta)=\hat A(\sigma+1, \zeta)$.
It follows  that the operator $\hat A(\sigma, \zeta)$ is naturally defined on functions in the  space of $L^2$ functions defined on $S^1$.  Through the identification $u(\sigma)= \tilde u(e^{\,2\pi i\sigma})$ we  consider this as a space of periodic functions on $[0,1]$ and denote it by  $L^2_{per}([0,1])$.  

Often one has to deal with operators that  are periodic with period $T>0$ that is not necessarily equal to $1$. It is elementary to modify our definitions of the Fourier-Laplace transform of a function and a linear operator in this case. For a given function $h$ and $T>0$ our objective is to obtain define the Fourier-Lapalace transform of $h$ which is periodic of period  $T$. We set $h_T(x)=h(Tx)$ and let naturally $\hat h(\xi, \zeta)=\hat h_T(\xi/T,\zeta)$ so that 
\[
\hat h(\xi, \zeta)=\sum_{-\infty<k<\infty} e^{\,-i(\xi +Tk)\zeta/T} h(\xi +Tk), \qquad  h(x)=\frac{1}{2\pi} \int_{\mu=0}^{2\pi} e^{\, i x \zeta/T} \hat h(x\modulo T, \zeta/T)\,d\mu,
\]
the Plancherel's formula is
\[
\int_{\mu=0}^{2\pi}\int_0^T |\hat h(\sigma,\zeta)|^2\, d\sigma\,d\mu=\int_\R |e^{\,\nu s/T} h(s)|^2\, ds, \qquad \zeta=\mu+i\nu,
\]
and the Fourier-Laplace transform of a $T$ periodic operator $A$ is
\[
\hat A(\xi,\zeta)=e^{\,-i\xi\zeta/T} A(\xi) e^{\,i\xi\zeta/T}.
\]
The operator $\hat A(\xi, \zeta)$ acts on a space of functions $L^2_{per}([0,T])$. Note that from the Plancherel's formula we see that if $h\in L_a^2(\R)_+$ and its Fourier-Laplace transform is $T$ periodic then it is natural to take $\zeta=\mu+i\nu=\mu+iTa$, $\mu\in (0, 2\pi)$ as the path of integration. 

In many applications, and this will be in particular  the case in our context, the family of operators $\hat A(\sigma, \zeta)$ is Fredholm and depends holomorphically on the variable $\zeta$. If this is the case one can use the analytic Fredholm theorem to conclude that either $\hat A(\sigma, \zeta)$ is nowhere  invertible or it is invertible in the set of all admissible $\zeta$ except possibly a discrete set.  If the latter happens then in order to solve the equation 
\[
A(x) h=g, 
\]
we can pass to  the Fourier-Laplace transform 
\begin{equation}
\label{fl inverse operator}
\hat A(\xi, \zeta)\hat h(\xi, \zeta)=\hat g(\xi, \zeta)\Longrightarrow h(x)=\frac{1}{2\pi} \int_{\mu=0}^{2\pi} e^{\, i x \zeta/T} (\hat A^{-1} \hat g) (x\modulo T, \zeta/T)\,d\mu,
\end{equation}
where in the last integral the path of integration should avoid the poles of $\hat A^{-1} (x\modulo T, \zeta)$. If between two such paths there is no pole of $\hat A^{-1} (x\modulo T, \zeta)$ then the path of integration can be shifted from one of the paths to the other horizontally  without changing the value of the integral. This follows by Cauchy's theorem, since the integrals over the vertical segments cancel out due to (\ref{fl quasi period}).  This means for instance that we can get the inverse of $A(x)$ in a space of functions $L^2_{a}(\R)_+$  whenever $\hat A^{-1}(\xi, \zeta)$ is analytic   in  some neighbourhood of the segment $\zeta=\mu+iT a$, $\mu\in [0, 2\pi]$.  Alternatively, this means that $\hat A^{-1} (\xi, \zeta)$ is well defined in the space $L^2_{per}([0,T])$, for $\zeta=\mu+iT\nu$, $|\nu-a|<\kappa$ with some $\kappa>0$. It may however happen that $\hat A^{-1} (\xi, \zeta)$ is analytic along two paths   $\zeta_j=\mu+iT\nu_j$, $j=1,2$, $\mu\in [0,2\pi]$ and $\nu_1<\nu_2$, but it has a pole at some $\zeta^*=\mu^*+i T\nu^*$, with $\nu_1<\nu^*<\nu_2$, $\mu^*\in (0,2\pi)$. In this case formula (\ref{fl inverse operator}) would give two solutions $h_1$ and $h_2$ (by integrating over the paths $\zeta=\mu+i T\nu_j$, $j=1,2$) which would differ by an element of the kernel of $A(x)$. This corresponds to the residue of $\hat A^{-1} (\xi, \zeta)$, $\zeta^* =\mu^*+iT\nu^*$.

\subsection{Mapping properties  of $L_{w_\tau}$ in weighted Sobolev spaces}

Going back to our context, we see that since $L_{w_\tau}$ is $T_\tau$ periodic in the $z$ variable, and so is  it induces a family of  operators  on $L^2_{\gamma, per}(\R^2\times [0,T_\tau])$, 
which is densely  defined and holomorphic, as a function of $\zeta$,  in a neighbourhood of the segment $[0,2\pi]$. Here and below $H^\ell_{\gamma,per}(\R^2\times [0,T_\tau])$ is a subspace of $H^\ell(\R^2\times [0,T_\tau])$ which consists of functions that are periodic in $z$ and whose grow (decay) away from $D_\tau$ is controlled by $e^{\, -\gamma\big(\frac{r-\rho_\tau(z)}{\ve}\big)}$, cf.  (\ref{metal urbain}). Later on we will also consider the space of functions $H^\ell_{\gamma,per}(\pD\times \R)$ consisting of functions defined on $\pD$ (here by $\pD$ we denote  a one period portion  of $D_\tau$ with the top and the bottom identified) and whose decay away from $\pD$ is controlled by $e^{\,-\gamma \st}$. These two norms are related locally, near $\pD$, by formulas analogous to (\ref{nick cave 1})--(\ref{nick cave 2}).

 If we restrict $L_{w_\tau}$  to the subspace of $L^2_{a,\gamma}{(\R^3)_+}$ of $L^2_{a, \gamma}(\R^3)$ functions  that  are supported in the set $z>-1$, and consider it  as acting on Fourier-Laplace transforms of such functions, then we can obtain a parametrix for the operator $L_{w_\tau}$ via the Fourier-Laplace inversion formula (\ref{fl inverse operator}).  As we pointed out earlier  the advantage in working with the family $\hat {L}_{w_\tau}(\zeta)$, is the fact that we can use the theory developed in  \cite{MR837256} and \cite{pacard_lecture}. 
 
Using the  Fourier-Laplace transform  we can consider the family of operators $\hat{L}_{w_\tau} (\zeta)$ instead of ${L}_{w_\tau}$. We will write the operator $\hat{L}_{w_\tau}(\zeta)$ in terms of variables $(x_1, x_2,  \xi)$ (here $\xi\in [0,T_\tau]$): 
\[
\begin{aligned}
\hat L_{w_\tau}(\zeta) &=\ve [\Delta+T_\tau^{-2}(\partial_{\xi\xi}+2 i\zeta \partial_{\xi}-\zeta^2)]+\frac{1}{\ve} f'(w_\tau)
\end{aligned}
\]
This operator is defined for functions in $H^2_{\gamma, per}(\R^2\times [0,T_\tau])$ and induces a densely defined operator on $L^2_{\gamma, per}(\R^2\times [0,T_\tau])$.  In order that the inversion formula for the Fourier-Laplace transform made sense we need to know the Fredholm property at least for $\zeta=\mu+i\nu$, where $\mu\in [0,2\pi]$ and $|\nu|$ is small, or in other words when $\zeta$ is in a neighbourhood of the segment $[0,2\pi]$. 
In order to prove that this operator is Fredholm we use the following:
\begin{lemma}\label{concentration lemma}
Let $A_R=\{({\tt x}, \xi)\mid r-\rho_\tau(\xi)\in (-R,R), r=|{\tt x}|=\sqrt{x_1^2+x_2^2}, \xi\in [0,T_\tau]\}$  and let $M>0$ be such that $f'(w_\tau)<-\frac{\sqrt{2}}{2}$ in $\R^2\times [0,T_\tau]\setminus A_{\ve M}$. There exists $\delta_\tau>0$ such that for all $\zeta=\mu+ia$, $\mu\in[0,2\pi]$, and $\gamma$ such that  $a^2+\gamma^2<\delta_\tau$,   and  all sufficiently small $\ve$,    it holds
\begin{equation}
\ve \|\nabla \phi\|^2_{L^2_{\gamma, per}(\R^2\times [0,T_\tau])}+\ve^{-1} \|\phi\|^2_{L^2_{\gamma, per}(\R^2\times [0,T_\tau])}\leq C\|\hat L_{w_\tau}(\zeta) \phi\|^2_{L^2_{\gamma, per}(\R^2\times [0,T_\tau])}+C\ve^{-1}\|\phi\|^2_{L^2(A_{\ve M})}.
\label{conc 1}
\end{equation}
for any function $\phi\in H^2_{\gamma, per}(\R^2\times [0,T_\tau])$. The constant $C$ above depends on $\zeta, M$ and $\gamma$. 
\end{lemma}
\begin{proof}[Proof of Lemma \ref{concentration lemma}]
This type of estimate is well known and it can be found for instance in \cite{10.2307/j.ctt13x1d8z}. We will outline the proof here (following the proof of a similar result in \cite{dkp-2009}). We agree that  $\Gamma$ is one of the functions
\[
\Gamma=e^{\, \gamma\big(\frac{r-\rho_\tau(\xi)}{\ve}\big)}, \qquad \Gamma=\cosh^\gamma\big(\frac{r-\rho_\tau(\xi)}{\ve}\big).
\]
We take a cutoff function $\chi_{\ve M}$ which is supported in the complement of the set $A_{\ve M/2}$ and is identically equal to $1$ in the complement of the set $A_{\ve M}$. Let us denote 
\[
\phi_\zeta=e^{\,i\zeta\xi/T_\tau}\phi,
\]
so that 
\[
\hat L_{w_\tau}\phi=e^{-\,i\zeta\xi/T_\tau}[\ve\Delta+\frac{1}{\ve} f'(w_\tau)]\phi_\zeta=g.
\]
Multiply the left hand side of the last  equation by  $\bar{\phi}\Gamma\chi^2_{\ve M}$ and integrate by parts.  This gives
\[
\begin{aligned}
\int_{\R^2\times [0,T_\tau]} \hat L_{w_\tau}(\zeta)  \phi \bar{\phi}\Gamma\chi^2_{\ve M}&=-\ve \int_{\R^2\times [0,T_\tau]} [|\nabla \phi_\zeta|^2 +\zeta^2|\phi_\zeta|^2]\Gamma\chi^2_{\ve M}+\frac{1}{\ve}\int_{\R^2\times [0,T_\tau]} f'(w_\tau)|\phi_\zeta|^2\Gamma\chi^2_{\ve M}\\
&\qquad -\ve \int_{\R^2\times [0,T_\tau]} \nabla  \phi_\zeta \bar{\phi}_\zeta\nabla (\Gamma\chi^2_{\ve M})
\end{aligned}
\]
Young's inequality gives for example
\[
\ve|\nabla \phi_\zeta \cdot  \nabla \Gamma\bar{\phi}_\zeta|\leq \ve \kappa  |\nabla\phi_\zeta|^2 \Gamma +\frac{\ve}{4\kappa} |\phi_\zeta|^2 \frac{|\nabla \Gamma|^2}{\Gamma}
\leq \ve \kappa |\nabla\phi_\zeta|^2 \Gamma +\frac{C\gamma}{4\ve\kappa} |\phi_\zeta|^2 \Gamma.
\]
Combining similar manipulations and adjusting the constants in the Young's inequality and the exponent $\gamma$ suitably we find 
\begin{equation}\label{decay estimate}
\ve \int_{\R^2\times [0,T_\tau]} |\nabla\phi_\zeta|^2\Gamma\chi_{\ve M}^2+\frac{1}{\ve}\int_{\R^2\times [0,T_\tau]} |\phi_\zeta|^2 \Gamma\chi_{\ve M}^2\leq C\int_{\R^2\times [0,T_\tau]} |g|^2\Gamma\chi_{\ve M}^2+C\ve \int_{\R^2\times [0,T_\tau]} |\phi_\zeta|^2|\nabla\chi_{\ve M}|^2.
\end{equation}
As $|\nabla\chi_{\ve M}|=\mathcal O(\ve^{-1})$ and
\[
|\phi_\zeta|=e^{\,\xi\mathrm{Im}\zeta/T_\tau}|\phi|, \quad |\nabla \phi|=|\nabla (e^{-\,i\zeta\xi/T_\tau}\phi_\zeta)|\leq e^{\,|\xi\mathrm{Im}\zeta|/T_\tau}|\nabla\phi_\zeta|+\frac{|\zeta|}{T_\tau}e^{\,|\xi\mathrm{Im}\zeta|/T_\tau}|\phi_\zeta|,
\]
the Lemma follows from this. 
\end{proof}

\begin{remark}\label{rem concentration estimate}
Estimate (\ref{decay estimate}) is of separate interest and it and its variants will be used  for instance when we analyse the operator ${L}_{w_\tau}$ below.  In particular we will need such a variant   in the proof  Lemma \ref{decay in r} (to follow).  
To explain this let us suppose that the weight function  $\Gamma$ depends on $z$ as well, say $\Gamma=(\cosh z)^a e^{\,\gamma\big(\frac{r-\rho_\tau(z)}{\ve}\big)}$ and consider the problem
\[
L_{w_\tau} \phi = g,
\]
where $\phi, g\in L^2_{a, \gamma}(\R^3)$. Choosing the cutoff function $\chi_{\ve M}$ as above (understood now as a function on $\R^3$) and multiplying by $\phi\Gamma\chi_{\ve M}^2$ we see that the term we need to control is of the form 
\[
\ve | \nabla \phi \cdot\nabla \Gamma \phi|\leq  \ve \kappa |\nabla\phi_\zeta|^2 \Gamma +\frac{C\gamma}{4\ve\kappa} |\phi_\zeta|^2 \Gamma,
\]
where the last inequality follows since we still have
\[
\frac{|\nabla \Gamma|}{\Gamma}\leq C\ve^{-1}.
\]
As a consequence we get an estimate of the same type as (\ref{decay estimate}) but with integrals taken over the whole space $\R^3$.
\end{remark}

}

{\color{black}{
\begin{lemma}\label{fredholm a}
The operator $\hat{L}_{w_\tau}(\zeta)$ acting on $H^2_{\gamma, per}(\R^2\times [0,T_\tau])$ is Fredholm.
\end{lemma}
\begin{proof}[Proof of Lemma \ref{fredholm a}]
We need to show that $\hat{L}_{w_\tau}(\zeta)$ has  finite dimensional kernel, closed range and that  codimension of the range is also finite. To see that the $\mathrm{dim}\,\mathrm{Ker}\,\hat{L}_{w_\tau}(\zeta)$ is finite we argue by contradiction. Using notation of Lemma \ref{concentration lemma} let 
\[
\mathcal B_1=\left\{\phi\in H^2_{\gamma, per}(\R^2\times [0,T_\tau])\mid  \hat L_{w_\tau}(\zeta) \phi=0, \quad \|\phi\|_{L^2(A_{\ve M})}=1\right\}.
\]
By Lemma \ref{concentration lemma} we know that  set $\mathcal B_1$ is bounded in $H^1_{\gamma,per}(\R^2\times [0,T_\tau])$ and then by Sobolev embedding it is compact in $L^2(A_{\ve M})$ and thus it must be finite dimensional.  To show that $\hat{L}_{w_\tau}(\zeta)$ has finite range we argue similarly (see for instance \cite{pacard_lecture} for a detailed proof).  To show that the codimension of the range is finite we use the fact that $\mathrm{dim}\,\mathrm{Ker}\,\hat{L}_{w_\tau}(\zeta)=\mathrm{codim}\, \mathrm{Range}(\hat L_{w_\tau}(\bar\zeta))$, by duality (the dual of $L^2_{\gamma, per}(\R^2\times [0,T_\tau])$ being $L^2_{-\gamma, per}(\R^2\times [0,T_\tau])$).  
\end{proof}

We will use this in proving:
\begin{proposition}\label{prop isomor}
There exists $\delta_\tau>0$ and a finite set $\mathcal S_0$, such that for all $a, \gamma$ with   $a^2+\gamma^2<\delta_\tau$, $a\notin \mathcal S_0$,  for all sufficiently small $\ve$ and for all $g\in L^2_{a,\gamma}(\R^3)$ there exists a solution of the problem
\begin{equation}\label{isomor 1}
L_{w_\tau}\phi=g,
\end{equation}
where $\phi\in H^2_{-|a|, \gamma}(\R^3)$.
\end{proposition}
Note that even if the right hand side of (\ref{isomor 1}) is decaying as $z\to \pm \infty$ (i.e. $a>0$) we get a solution which in general may be increasing as $z\to\pm\infty$ at the exponential rate proportional to $e^{\,|a||z|}$. 

}}

\begin{proof}[Proof of Proposition \ref{prop isomor}]
{\blue{The idea of the proof is to show that $\hat{L}_{w_\tau}(\zeta)$ is an isomorphism for $\zeta$ in some neighbourhood of $[0, 2\pi]$, except possibly a finite set of points, and then use the parametrix formula to solve (\ref{isomor 1}). Since $\hat{L}_{w_\tau}(\zeta)$ is a Fredholm family of holomorphic operators in an open set  $\mathcal U\subset \C$ with  $[0,2\pi]\subset \mathcal U$ it is either non invertible everywhere in $\mathcal U$ or    it is invertible except a discrete subset of $\mathcal U$ \cite{reed-simon_4}. In particular, if we consider $\zeta\in [0,2\pi]$  (note that the operator $\hat L_{w_\tau}(\zeta)$ is self adjoint for $\zeta\in \R$)  and are able to show  that it is injective there except possibly a discrete set of points then we will conclude that it is  invertible in $[0, 2\pi]$ except the discrete set and then the same will be true at least in a neighbourhood $\mathcal U$ of this segment.

To carry out this plan  we consider  $\hat L_{w_\tau}$ taken with respect to variable $z$. This operator is defined on the space of functions in  $H^2_{per}(\R^2\times [0, T_\tau])$ which consists of functions which are periodic with period $T_\tau$.  
 Recall  that we have
\[
\hat L_{w_\tau}(\zeta)=e^{\,-i\zeta \xi/T_\tau}\{\ve \Delta +\ve^{-1} f'(w_\tau)\} e^{\,i\zeta \xi/T_\tau}.
\]
We want to express $\hat L_{w_\tau}$ in terms of the stretched Fermi co-ordinates in $\mathcal N_\delta$. Let $\pD$ be the one period  piece of $D_\tau$ (i.e. $0\leq z\leq T_\tau$)  with the top and the bottom identified.  The natural domain for the expression of the Fourier-Laplace transforms  of functions in $L_{per}^2(\R^2\times [0,T_\tau])$ in the  stretched Fermi coordinates is $\pD\times [-\delta/\ve, \delta/\ve]$. For example 
from the definition of the shifted Fermi coordinates we see that  
\[
(Y^*_\ve \xi)(\sy, \st)=\left[\sy+\ve \st  N_\tau(\sy)\right]\cdot {\tt e}_3.
\]
It is convenient to extend this  function from  $\pD\times [-\delta/\ve, \delta/\ve]$ to $\pD\times\R$. We will use for this purpose the cutoff function $\chi_{\delta/\ve}$  defined in (\ref{def cutoff}) and  set
\[
\xi^*= \chi_{\ve/\delta}(Y^*_\ve \xi)+(1-\chi_{\ve/\delta}) \sy\cdot {\tt e}_3
\]
for the extension of  $(Y^*_\ve \xi)=\xi^*$, understanding that this is a function of $(\sy, \st)$.  

We use the operator $\mathbb L_{w_\tau}$  (see (\ref{ext l wtau}))  to define also a natural extension of $\hat L_{w_\tau}$ to $\pD\times \R$
\[
\tilde {\mathbb L}_{w_\tau}(\zeta)= e^{\,-i\zeta \xi^*/T_\tau}\mathbb L_{w_\tau} e^{\,i\zeta \xi^*/T_\tau}.
\]
The operator $\tilde {\mathbb L}_{w_\tau}(\zeta)$ is "almost"  the Fourier-Laplace transform of $\mathbb L_{w_\tau}$.  Note that $\tilde {\mathbb L}_{w_\tau}(0)=\mathbb L_{w_\tau}$.  The strategy of the proof is to show first  that the operator $\tilde {\mathbb L}_{w_\tau}(\zeta)$ is injective and then conclude from this that 
$\hat L_{w_\tau}(\zeta)$ is injective. 

We will study the kernel of $\tilde{\mathbb L}_{w_\tau}(\zeta)$ in the space of functions $L_{\gamma, per}^2(\pD\times\R)$.  Let us suppose that for some $\gamma$, $|\gamma|<\delta_\tau$, and $\zeta\in [0,2\pi]$ there exists a function $\phi\in H^2_{\gamma,per}(\pD\times \R)$ 
\[
\tilde{\mathbb L}_{w_\tau}(\zeta) \phi_0=\mathbb L_{w_\tau} (e^{\,i\zeta \xi^*/T_\tau}\phi_0)=\mathbb L_{w_\tau}\phi_{0\zeta} =0,
\]
where we have denoted
\begin{equation}
\label{def phizeta}
\phi_{0\zeta} =e^{\,i\zeta \xi^*/T_\tau} \phi_0.
\end{equation}
We can normalize $\|\phi_{0\zeta}\|_{L^2_\gamma(\pD\times \R)}=1$ and then by elliptic estimates for any $M>0$   in the set $\pD\times (-M, M)$ the function $\phi_{0\zeta}$ is bounded (we bound the real and imaginary parts of $\phi_{0\zeta}$ separately). Take $M$ large so that $f'(w_\tau)<-2+\eta$ with some small $\eta>0$.  Take $\delta_\tau$ in the statement of the Proposition small so that $\gamma\in (-\sqrt{2-\eta}, \sqrt{2-\eta})$. Using the comparison principle for the operator $\tilde{\mathbb L}_{w_\tau}$ it is then easy to show that in fact 
\begin{equation}
\label{est exp phi}
|\phi_{0\zeta} (\sy, \st)|=|\phi_0|\leq Ce^{-\sqrt{2-\eta} |\st|}
\end{equation}
and therefore $\phi_0\in H^2(\pD \times\R)$. 

For complex valued functions   $\phi_1,\phi_2\in L^2(\pD \times\R)$ we define  Hermitian inner product
\[
\langle \phi_1, \phi_2\rangle=\int_{\pD\times \R} \phi_1\bar \phi_2 d V_{\pD}d\st 
\]
}}
{\blue{

Above $\nabla_{\pD}$  and $d V_{\pD}$ are  respectively  the gradient and  the volume element on $\pD$.
We introduce an orthogonal decomposition in $L^2(\pD \times\R)$ as follows:  Let ${\tt V}$ be the function defined in (\ref{def vv}) (we recall that it is an extension of $\partial_{\st} Y^*_\ve w_\tau$).   Given a function $\phi\in   L^2(\pD \times\R)$ we denote $\phi_\zeta=e^{\,i\zeta \xi^*/T_\tau} \phi$ and  decompose  
\[
\phi_{\zeta}=\phi_{\zeta}^\parallel + \psi_{\zeta} {\tt V}, 
\]
where 
\[
\phi_{\zeta}^\parallel \in \mathcal X_\gamma:=\left \{\phi\in L^2_{\gamma,per}(\pD\times \R)\left| \int_\R \phi{\tt V}\, d\st=\int_\R \bar \phi{\tt V}\, d\st =0\right.\right\}, \quad  \psi_{\zeta}=\frac{\int_\R  \phi_{\zeta} {\tt V}\, d\st}{\int_\R  {\tt V}^2\, d\st}.
\]
In particular  for $\phi_0\in \mathrm{Ker}\, \tilde{\mathbb L}_{w_\tau}(\zeta)$ we have
\begin{equation}
\label{sakamoto}
{\mathbb L}_{w_\tau}\phi_{0\zeta}={\mathbb L}_{w_\tau}\phi_{0\zeta}^\parallel+{\mathbb L}_{w_\tau}(\psi_{0\zeta} {\tt V})=0
\end{equation}
and 
\begin{equation}
\left\langle -{\mathbb L}_{w_\tau}\phi_{0\zeta}^\parallel, \phi_{0\zeta}^\parallel\right\rangle=-\left \langle {\mathbb L}_{w_\tau} \psi_{0\zeta} {\tt V}, \phi_{0\zeta}^\parallel\right\rangle.
\label{urinals}
\end{equation}
We will use this identity to estimate $\phi_{0\zeta}^\parallel$ in terms of suitable norm of $\psi_{0\zeta}$. To do so we need:
\begin{lemma}\label{coerc}
It holds 
\begin{equation}
\label{fl 2}
\left|\langle -{\mathbb L}_{w_\tau}\phi_{\zeta}^\parallel, \phi_{\zeta}^\parallel\rangle\right|\geq \frac{C}{\ve}\left(\|\partial_{\st}\phi_{\zeta}^\parallel\|^2_{L^2(\pD\times \R)}+\|\phi_{\zeta}^\parallel\|^2_{L^2(\pD\times \R)}\right)+C\ve \|\nabla_{\pD} \phi_{\zeta}^\parallel\|^2_{L^2(\pD\times \R)}.
\end{equation}
\end{lemma}
\begin{proof}[Proof of Lemma \ref{coerc}]
We recall the well known fact: with  $\varTheta(x)=\tanh\big(\frac{x}{\sqrt{2}}\big)$  the bilinear form 
\begin{equation}
\int_\R |\psi'|^2-f'(\varTheta)\psi^2
\label{coerc zero form}
\end{equation}
is positive definite on the space  of functions $L^2(\R)$ orthogonal to $\varTheta'(x)$. Consider a quadratic  form
\[
\mathcal B(\phi,\phi)=\frac{1}{\ve}\int_{\pD \times \R} |\partial_\st \phi|^2+\ve^2|\nabla_{\pD}\phi|^2-f'(\varTheta)|\phi|^2
\]
for  $\phi\in \mathcal X_\gamma$. Write
\[
\phi=\phi_1+\phi_2 \varTheta', \qquad \mbox{where}\qquad \left( \phi_1, \varTheta'\right)=0, \qquad \phi_2=\frac{\left( \phi, \varTheta'\right)}{\left( \varTheta', \varTheta'\right)},
\]
and where we have denoted 
\[
\left (\phi, \psi\right)=\int_\R \phi\bar \psi\,d\st
\]
We have 
\[
0=\left( \phi,   {\tt V}\right) =\left( \phi, \varTheta'\right) +\left(  \phi, {\tt V}-\varTheta'\right)=\phi_2\left( \varTheta', \varTheta'\right)+\left(  \phi, {\tt V}-\varTheta'\right), 
\]
and also 
\[
0=\nabla_{\pD}\left(\phi, {\tt V}\right)=\left( \nabla_\pD\phi, {\tt V}\right)+\left(\phi, \nabla_{\pD} {\tt V}\right)=\nabla_{\pD}\phi_2 \left( \varTheta', \varTheta'\right)^{2}+
\left( \nabla_\pD\phi, {\tt V}-\varTheta'\right)+\left(\phi, \nabla_{\pD} {\tt V}\right)
\]
Since 
\[
{\tt V}-\varTheta'=U'-\varTheta'+\mathcal O_{\mathcal C^{1,\alpha}_\mu\pD \times \R} (\ve^{2-\alpha})=\mathcal O_{\mathcal C^{1,\alpha}_\mu\pD \times \R} (\ve), \qquad \nabla_{\pD} {\tt V}=\mathcal O_{\mathcal C^{1,\alpha}_\mu\pD \times \R} (\ve^{2-\alpha})
\]
we get
\begin{equation}
\label{est phi2}
\|\phi_2 \varTheta'\|_{H^1(\pD \times \R)}\leq C\ve \|\phi\|_{H^1(\pD \times \R)}.
\end{equation}
By (\ref{coerc zero form})
\[
\mathcal B(\phi_1, \phi_1)
\geq \frac{C}{\ve}\left (\|\phi_1\|^2_{L^2(\pD \times \R)}+\|\partial_{\st}\phi_1\|^2_{L^2(\pD \times \R)}+\ve^2\|\nabla_{\pD}\phi_1\|^2_{L^2(\pD \times \R)}\right),
\]
hence  from   (\ref{est phi2})
\[
\mathcal B(\phi, \phi)
\geq \frac{C}{\ve}\left (\|\phi\|^2_{L^2(\pD \times \R)}+\|\partial_{\st}\phi\|^2_{L^2(\pD \times \R)}+\ve^2\|\nabla_{\pD}\phi\|^2_{L^2(\pD \times \R)}\right),
\]
for any $\phi\in \mathcal X_\gamma$.  

We get
\[
\begin{aligned}
\langle -{\mathbb L}_{w_\tau}\phi_{\zeta}^\parallel, \phi_{\zeta}^\parallel\rangle&=\mathcal B(\phi_{\zeta}^\parallel, \phi_{\zeta}^\parallel)-\frac{1}{\ve}\int_{\pD\times \R} [f'({\tt w}_\tau)-f'(\varTheta)]|\phi_{\zeta}|^2\,d V_{\pD}d\st\\
&\qquad -\frac{\ve}{2} \int_{\pD\times \R}\partial_\st\left(\st\mathbb \chi_{\ve/\delta}\right)|A_{\pD}|^2|\phi_{\zeta}|^2\, d V_{\pD}d\st 
+ \langle \chi_{\ve/\delta} \mathbb Q_\ve \partial_\st \phi_\zeta, \phi_\zeta\rangle  +\ve \langle \chi_{\ve/\delta}\mathbb A_\ve\phi_\zeta, \phi_\zeta\rangle\\
&=\mathcal B(\phi_{\zeta}^\parallel, \phi_{\zeta}^\parallel)+\left (\mathcal O(1)+\mathcal O(\ve/\delta)+\mathcal O(\delta^2)\right)\left(\|\phi_{\zeta}^\parallel\|^2_{L^2(\pD \times \R)}+\|\partial_{\st}\phi_{\zeta}^\parallel\|^2_{L^2(\pD \times \R)}\right)+\mathcal O(\ve \delta)\|\nabla_{\pD}\phi_{\zeta}^\parallel\|_{L^2(\pD \times \R)}^2.
\end{aligned}
\]
Since $\delta$ can be taken as small as we wish the assertion of the Lemma follows. 
\end{proof}

Now, we need to control the mixed  term in (\ref{urinals})
\[
\begin{aligned}
\langle - {\mathbb L}_{w_\tau} (\psi_{0\zeta} {\tt V}), \phi_{0\zeta}^\parallel\rangle&=
-\ve\left \langle {\tt V} \left(\mathcal J_{D_\tau}+\chi_{\ve/\delta}\mathbb A_\ve\right)\psi_{0\zeta},\phi^\parallel_{0\zeta} \right\rangle - \ve \left\langle \left[\Delta_{D_\tau} +\chi_{\ve/\delta}\mathbb A_\ve, {\tt V}\right]\psi_{0\zeta},  \phi^\parallel_{0\zeta} \right\rangle\\
&\qquad +\left\langle \mathcal O_{C^{0, \alpha}_\mu(D_\tau\times \R)} (\ve^{3-\alpha})\psi_{0\zeta}, \phi^\parallel_{0\zeta} \right\rangle
\\
&=\left\langle \mathcal O_{C^{0, \alpha}_\mu(D_\tau\times \R)} (\ve^{2})\nabla_{\pD}\psi_{0\zeta}, \nabla_{\pD}\phi^\parallel_{0\zeta} \right\rangle+\left\langle \mathcal O_{C^{0, \alpha}_\mu(D_\tau\times \R)} (\ve^{3-\alpha})\psi_{0\zeta}, \phi^\parallel_{0\zeta} \right\rangle
\end{aligned}
\]
where the last equality follows because the coefficients of the operator $\mathbb A_\ve$ are bounded by $\ve\st$ and ${\tt V}$ is exponentially decaying in $\st$. 
By the  Cauchy-Schwarz inequality for any some $\eta>0$ small we get 
\begin{equation}
\label{fl 5}
\left|\left\langle - {\mathbb L}_{w_\tau}(\psi_{0\zeta} {\tt V}), \phi_{0\zeta}^\parallel\right\rangle\right|\leq  \eta \left(\ve^{-1}\|\phi_{0\zeta}^\parallel\|^2_{H^1(\pD\times \R)}+ \ve\|\phi_{0\zeta}^\parallel\|^2_{H^1(\pD\times \R)}\right)+C_\eta \ve^3\|\psi_{0\zeta}\|^2_{H^1(\pD)}.
\end{equation}
It follows from (\ref{urinals}) and Lemma \ref{coerc}
\[
\frac{1}{\ve}\left(\|\partial_{\st}\phi_{0\zeta}^\parallel\|^2_{L^2(\pD\times \R)}+\|\phi_{0\zeta}^\parallel\|^2_{L^2(\pD\times \R)}\right)+\ve \|\nabla_{\pD} \phi_{0\zeta}^\parallel\|^2_{L^2(\pD\times \R)}\leq C\ve^{3}\|\psi_{0\zeta}\|^2_{H^1(\pD)}.
\]

Now consider the orthogonal complement of $\mathcal X_\gamma$. From (\ref{expr v3})   we obtain:  
\[
\mathbb L_{w_\tau} (\psi_{0\zeta}{\tt V})=\ve{\tt V} \left(\mathcal J_{\pD}+\chi_{\ve/\delta}\mathbb A_\ve\right)\psi_{0\zeta}+ \ve \left[\Delta_{D_\tau} +\chi_{\ve/\delta}\mathbb A_\ve, {\tt V}\right]\psi_{0\zeta}+\mathcal O_{C^{0, \alpha}_\mu(D_\tau\times \R)} (\ve^{3-\alpha})\psi_{0\zeta},
\]
Using this and projecting (\ref{sakamoto}) onto ${\tt V}$ and integrating over $\R$ we get
\begin{equation}
\label{psi0}
\mathcal J_{\pD} \psi_{0\zeta} = T(\psi_{0\zeta}, \phi_{0\zeta})
\end{equation}
where
\begin{equation}
\label{psi00}
\begin{aligned}
\|T(\psi_{0\zeta}, \phi_{0\zeta})\|_{L^2(\pD)}&\leq C\delta\|\psi_{0\zeta}\|_{H^2(\pD)}+C\left( \ve^{-1} \|\phi_{0\zeta}^\parallel\|_{L^2(\pD\times\R)}+\ve^{2-\alpha}\|\phi_{0\zeta}^\parallel\|_{H^1(\pD\times\R)}\right)\\
&\leq C\left(\delta \|\psi_{0\zeta}\|_{H^2(\pD)} +\ve\|\psi_{0\zeta}\|_{H^1(\pD)} \right)
\end{aligned}
\end{equation}
We claim that from this it follows that for any $\zeta\in (0,1)$ there exists $\ve_\zeta>0$ such that for any $\ve\in (0, \ve_\zeta)$ we have $\psi_{0\zeta}=0$ and hence $\phi_0=0$. To show this claim  we note  that by definition 
\[
\begin{aligned}
\psi_{0\zeta}&=\frac{\left(\phi_{0\zeta}, {\tt V}\right)}{\left({\tt V}, {\tt V}\right)}=\frac{\left(e^{\,i\zeta \xi^*/T_\tau}\phi_0 , {\tt V}\right)}{\left({\tt V}, {\tt V}\right)}
=\frac{\left(e^{\,i\zeta (\sy_3+\ve\chi_{\ve/\delta} \st  N_\tau\cdot {\tt e}_3)/T_\tau}\phi_0 , {\tt V}\right)}{\left({\tt V}, {\tt V}\right)}=\frac{e^{\,i\zeta\sy_3/T_\tau} \left(e^{\,i\zeta\ve\chi_{\ve/\delta} \st  N_\tau\cdot {\tt e}_3)/T_\tau}\phi_0 , {\tt V}\right)}{\left({\tt V}, {\tt V}\right)}\\
&=e^{\,i\zeta\sy_3/T_\tau}\tilde \psi_{0\zeta}
\end{aligned}
\]
where $\tilde \psi_{0\zeta}$ is periodic in $\sy_3$ with period $T_\tau$.
We see that $\psi_{0\zeta}$  satisfies 
\[
\psi_{0\zeta}(\sy_1, \sy_2, \sy_3+T_\tau)=e^{\,i\zeta} \psi_{0\zeta}(\sy_1, \sy_2, \sy_3), \qquad \sy = (\sy_1, \sy_2, \sy_3)\in D_\tau,
\]  
with similar relation for $\partial_{\sy_3} \psi_{0\zeta}$. By Proposition 4.2   in \cite{MR1941630} we know that the operator $J_{\pD}$ is invertible in the space of functions satisfying these conditions as long as $\zeta\in (0, 2\pi)$ with an inverse whose norm depends on $\tau$. The claim now follows from (\ref{psi0}) and (\ref{psi00}).

In particular we conclude that the operator $\tilde{\mathbb L}_{w_\tau}(\zeta)$ is injective for $\zeta\in (0, 2\pi)$ and by the same argument for  $\zeta\in (-2\pi, 0)$ (note that $\tilde{\mathbb L}^*_{w_\tau}(\zeta)= \tilde{\mathbb L}_{w_\tau}(-\zeta)$). A version of Lemma  \ref{concentration lemma}  for  $\tilde{\mathbb  L}_{w_\tau}(\zeta)$ shows that this operator  is Fredholm, depends analytically on $\zeta$ and, as a consequence, it is invertible in a neighbourhood of $[0,2\pi]$ except for a discrete set. 

Now let us suppose that for some $\zeta\in (0,2\pi)$ there exists a function $\phi_0\in H^2_{\gamma, per}(\R^2\times [0,T_\tau])$, with some $\gamma$,  $|\gamma|$  small,   such that $\hat L_{w_\tau}(\zeta) \phi_0=0$. Since $\phi_0$ is bounded locally near $\pD$ we can use comparison principle to show that $\phi_0$ is decaying away from $\pD$ 
at least like $e^{\,-\sqrt{2-\eta}\frac{|r-\rho(\xi)|}{\ve}}$ (the argument is similar to the one leading to (\ref{est exp phi})). Using Lemma \ref{concentration lemma} we get
\begin{equation}
\label{est h2l2}
\|\phi_0\|_{H^2_{\pm \gamma, per}(\R^2\times [0, T_\tau])}\leq C\ve^{-1}\|\phi_0\|_{L^2_{per}(\R^2\times [0, T_\tau])}.
\end{equation} 
We normalize $\|\phi_0\|_{H^1_{per}(\R^2\times [0, T_\tau])}=1$ and  set $\tilde \phi_0=\chi_{\ve/\delta}Y_\ve^*\phi_0$. With this notation (see (\ref{nick cave 1}))
\[
\|\tilde \phi_0\|_{L^2(\pD\times \R)}\sim \ve^{-1/2} \|\phi_0\|_{L^2(\{\mathrm{dist}\,(x, \pD)<\delta\})}
\]
since $d\st dV_{\pD}\sim \ve^{-1} dx$. Similarly, we have
\[
\|\tilde \phi_0\|_{H^1(\pD\times \R)}\sim \ve^{1/2} \|\phi_0\|_{H^1(\{\mathrm{dist}\,(x, \pD)<\delta\})}.
\] 
Next, we observe that since $\phi_0$ is decaying exponentially away from $\pD$ we have by (\ref{est h2l2})
\[
\begin{aligned}
\|\phi_0\|^2_{H^1_{per}(\R^2\times [0,T_\tau])}&=\|\phi_0\|^2_{H^1(\{\mathrm{dist}\,(x, \pD)<\delta\})}+\|\phi_0\|^2_{H^1(\{\mathrm{dist}\,(x, \pD)>\delta\})}\\
&\leq \|\phi_0\|^2_{H^1(\{\mathrm{dist}\,(x, \pD)<\delta\})}+\mathcal O(e^{\,-c\delta/\ve})\|\phi\|^2_{H^1_{\gamma, per}(\R^2\times [0, T_\tau])\cap H^1_{-\gamma, per}(\R^2\times [0, T_\tau]) }\\
&\leq \|\phi_0\|^2_{H^1(\{\mathrm{dist}\,(x, \pD)<\delta\})}+\mathcal O(\ve^{-1}e^{\,-c\delta/\ve})\|\phi_0\|^2_{H^1_{per}(\R^2\times [0,T_\tau])},
\end{aligned}
\]
hence
\begin{equation}
\label{low phi}
 \|\phi_0\|^2_{H^1(\{\mathrm{dist}\,(x, \pD)<\delta\})}\geq \frac{1}{2}\|\phi_0\|^2_{H^1_{per}(\R^2\times [0,T_\tau])}=\frac{1}{2}.
 \end{equation}
Given all this we claim that  we can find a nontrivial function $\tilde \phi=\tilde \phi_0+\tilde \phi_1$, $\psi\in L^2(\pD\times \R)$,  such that $\tilde{\mathbb L}_{w_\tau}(\zeta)  \tilde\phi=0$, by solving 
\begin{equation}
\label{till}
\tilde{\mathbb L}_{w_\tau}(\zeta) \tilde\phi_1=-\left[\chi_{\ve/\delta}, \tilde {\mathbb L}_{w_\tau}(\zeta)\right]\tilde \phi_0
-e^{\,-i\zeta\xi^*/T_\tau}\chi_{\ve/\delta}(1-\chi_{\ve/\delta})\left[\left(\ve \st|A_{D_\tau}|^2+ \mathbb Q_\ve\right)\partial_\st+\ve\mathbb A_\ve\right]e^{\,i\zeta\xi^*/T_\tau}\tilde \phi_0:= R_{\ve/\delta}(\sy, \st).
\end{equation}
In fact, since $R_{\ve/\delta}$  is supported in the set $\delta/2\ve\leq |\st|\leq \delta/\ve$ therefore
\[
\|R_{\ve/\delta}\|_{L^2(\pD\times \R)}\leq C e^{\,-c\delta/\ve}\|\phi_0\|_{H^1_{\gamma,per}(\R^2\times [0,T_\tau])\cap H^1_{-\gamma,per}(\R^2\times [0,T_\tau])}.
\]
Next we decompose $\tilde \phi_1=\tilde\phi_1^\parallel+\tilde \psi_1 {\tt V}$ and use (with only slight modifications) the argument that we have used to show that $\tilde {\mathbb L}_{w_\tau}(\zeta)$ is injective to get:
\[
\|\tilde \phi_1\|_{H^1(\pD\times \R)}\leq C\ve^{-1} \|R_{\ve/\delta}\|_{L^2(\pD\times \R)}\leq C e^{\,-c\delta/\ve}\|\phi_0\|_{L^2_{per}(\R^2\times [0, T_\tau])}.
\]
From (\ref{low phi}) it now  follows 
\[
\begin{aligned}
\|\tilde \phi\|_{H^1(\pD\times \R)}&\geq \|\tilde \phi_0\|_{H^1(\pD\times \R)}-\|\tilde\phi_1\|_{H^1(\pD\times \R)}\\
&\geq C\ve^{1/2} \|\phi_0\|^2_{H^1(\{\mathrm{dist}\,(x, \pD)<\delta\})}+\mathcal O(e^{\,-c\delta/\ve})\|\phi_0\|_{L^2_{per}(\R^2\times [0, T_\tau])}>0,
\end{aligned}
\]
for $\ve$ sufficiently small. This  contradicts the fact that   $\tilde{\mathbb L}_{w_\tau}(\zeta)$ is injective. 
Taking this  into account we see that $\hat{L}_{w_\tau}(\zeta)$ is invertible at least for $\zeta\in (0,2\pi)$, and thus by the Fredholm alternative is invertible for all $\zeta$ such that $|\mathrm{Im}\, \zeta|<\delta_\tau$, expect possibly a finite set  where $\hat{L}^{-1}_{w_\tau}(\zeta)$ has poles. We claim that the required properties of $L_{w_\tau}$ follow now by taking the inverse Fourier-Laplace transform at any $a$ for which $\hat{L}^{-1}_{w_\tau, \gamma}(\zeta)$ is well defined for $\zeta=\mu+i T_\tau a$, $\mu\in [0,2\pi]$. Indeed, given $g\in L^2_{a, \gamma}(\R^3)$ with $a^2+\gamma^2<\delta_\tau$ and cutoff functions $\chi^\pm(z)$ such that $\chi^+(z)+\chi^-(z)=1$ and $\mathrm{supp}\, \chi^+=(-1, \infty)$ we can solve
\[
L_{w_\tau}\phi^{\pm}= \chi^\pm g. 
\]
To do this we let $\hat g^\pm$ to be   the Fourier-Laplace  transforms of $g^\pm$. Then we  solve
\[
\hat L_{w_\tau}(\zeta)\hat\phi^\pm = \hat g^\pm\Longrightarrow \hat \phi^\pm =\hat L_{w_\tau}(\zeta)^{-1}\hat g^\pm,
\]
and by taking the inverse of the Fourier-Laplace transform $\mathcal F$ we determine
\[
\phi^\pm =\mathcal F^{-1}\big(\hat L_{w_\tau}(\zeta)^{-1}\hat g^\pm\big),
\] 
and define 
\[
\phi=G_{w_\tau}(g):=\phi^-+\phi^+.
\]
This ends the proof.

}}
\end{proof}

{\blue{
\begin{remark}\label{bootstrap}
We will describe a useful  consequence of local elliptic estimates. Let us suppose that we know {\it a priori} $\phi, g\in L^2_{a, \gamma}(\R^3)$ where
\[
\Delta\phi =g.
\]
The goal is to obtain weighted Sobolev estimates for the derivatives of $\phi$. First, consider a cube $Q_{r}(x_0)$ centred at $x_0\in \R^3$ and with its sides equal to $r$. Standard elliptic estimates show
\[
r\|D^2\phi\|_{L^2(Q_r(x_0))}+\|\nabla\phi\|_{L^2(Q_r(x_0))}\leq C\|g\|_{L^2(Q_{2r}(x_0))}+Cr^{-1}\|\phi\|_{L^2(Q_{2r}(x_0))}.
\]  
If $r=\ve$ then we get from this
\[
\ve\|D^2\phi\|_{L^2_{a,\gamma}(Q_{\ve}(x_0))}+\|\nabla\phi\|_{L^2_{a,\gamma}(Q_\ve(x_0))}\leq C\|g\|_{L^2_{a,\gamma}(Q_{2\ve}(x_0))}+C\ve^{-1}\|\phi\|_{L^2_{a,\gamma}(Q_{2\ve}(x_0))},
\]
since the exponential weights are comparable on the sets with diameters proportional to $\ve$. Arranging now a countable collection of cubes $\{Q_{\ve}(x_j)\}_{j\in \mathbb N}$ in such a way that for each $x_j$ the number of cubes $Q_{2\ve}(x_{j'})$, $j'\neq j$,  whose intersection with $Q_{\ve}(x_j)$ is nonempty is finite and  bounded independently on $j$, while at the same time $\R^3=\cup_{j\in \mathbb N} Q_{\ve}(x_j)$, we see that above local estimates can be summed up  to yield:
\begin{equation}\label{boostrap 1}
\ve\|D^2\phi\|_{L^2_{a,\gamma}(\R^3)}+\|\nabla\phi\|_{L^2_{a,\gamma}(\R^3)}\leq C\|g\|_{L^2_{a,\gamma}(\R^3)}+C\ve^{-1}\|\phi\|_{L^2_{a,\gamma}(\R^3)}.
\end{equation} 
\end{remark}

\begin{lemma}\label{decay in r}
Let $\phi\in L^2_{a,\gamma'}(\R^3)$ be a solution of $L_{w_\tau} \phi=g$ with $g\in L^2_{a, \gamma}(\R^3)$ where $\gamma>0$, $\gamma'<\gamma$ and $a^2+\gamma^2<\delta_\tau$, $a^2 + \gamma'^2<\delta_\tau$. Then $\phi\in L^2_{a, \gamma}(\R^3)$.   
An analogous  statement holds when we assume that $\gamma<0$ and $\gamma<\gamma'$. 
\end{lemma}
\begin{proof}
We follow the proof of a similar result in \cite{dkp-2009}.  Let $\chi$ be a cutoff function supported in the set $\ve M<r-\rho_\tau(z)$, $r^2=x_1^2+x_2^2$,  and such that $\chi\equiv 1$ in the set  $r-\rho_\tau(z)>2\ve M$ where $M$ is chosen so  that $f'(w_\tau)<-1$ for $r-\rho_\tau(z)>\ve M$. We calculate 
\[
L_{w_\tau} (\chi \phi)=\chi g+[\ve \Delta,\chi]\phi\equiv g_1, \quad [\ve\Delta,\chi]\phi=\ve\Delta(\chi\phi)-\ve\chi\Delta\phi.
\]
We have $\ve\nabla\chi=\mathcal O(1)$ and $\ve\Delta\chi=\mathcal O(\ve^{-1})$. Moreover, by local elliptic estimates applied to the equation
\[
\Delta \phi=\ve^{-1} g+\ve^{-2}f'(w_\tau)\phi
\]
we can show that (see Remark \ref{bootstrap})
\[
\|\ve\nabla\chi\cdot\nabla\phi\|_{L^2_{a,\gamma}(\R^3)}\leq C\ve^{-1}\|g\|_{L^2_{a, \gamma}(\R^3)}+C\ve^{-2}\|\phi\|_{L^2_{a, \gamma'}(\R^3)}.
\]
We find from this
\[
\|\ve\nabla\chi\cdot\nabla\phi\|_{L^2_{a,\gamma}(\R^3)}+\|\ve \phi\Delta\chi\|_{L^2_{a,\gamma}(\R^3)}\leq C\ve^{-1}\|g\|_{L^2_{a, \gamma}(\R^3)}\|+C\ve^{-2}\|\phi\|_{L^2_{a, \gamma'}(\R^3)}.
\] 
Above, we use the fact that the weighed norms $\|\cdot\|_{L^2_{a,\gamma}(\R^3)}$ and $\|\cdot\|_{L^2_{a,\gamma'}(\R^3)}$ are comparable in the set $r-\rho_\tau(z)\in [\ve M, 2\ve M]$.  
From this we obtain 
\[
\|g_1\|_{L^2_{a, \gamma}(\R^3)}\leq C\ve^{-1}\|g\|_{L^2_{a, \gamma}(\R^3)}+C\ve^{-2}\|\phi\|_{L^2_{a, \gamma'}(\R^3)}.
\]
Now we  solve the problem
\[
\begin{aligned}
L_{w_\tau} \phi_{1,R}&=g_1, \quad &\mbox{in}\ \Omega_{\ve M, R},\\
\phi_{1,R}&=0, \quad &\mbox{on}\ \Omega_{\ve M, R},
\end{aligned}
\]
in a bounded  set $\Omega_{\ve M, R}=\{\ve M<r-\rho_\tau(z)< R, |z|<R\}$. Using similar argument as the one leading to  (\ref{decay estimate}) in the proof of Lemma \ref{concentration lemma} we get
\[
\begin{aligned}
\|\phi_{1,R}\|_{H^1_{a, \gamma}(\Omega_{\ve M, R})}&\leq C\ve^{-1/2}\|g_1\|_{L^2_{a, \gamma}(\R^3)}\\
&\leq C\ve^{-3/2}\|g\|_{L^2_{a, \gamma}(\R^3)}+C\ve^{-5/2}\|\phi\|_{L^2_{a, \gamma'}(\R^3)}.
\end{aligned}
\] 
Note the that in the first of the above inequalities only the right hand side of the equation appears, which  is due to the fact that we assumed homogeneous Dirichlet boundary conditions on $\phi_{1,R}$ and we do not need to introduce the cut off function $\chi_{\ve M}$ in proving a version of  Lemma \ref{concentration lemma} needed here.
Letting $R\to \infty$ we get a solution $\phi_1$ of the equation $L_{w_\tau}\phi_1=g_1$ but now in the set $\ve M<r-\rho_\tau(z)$, such that 
\[
\|\phi_{1}\|_{H^1_{a, \gamma}(\R^3)}\leq C\ve^{-3/2}\|g\|_{L^2_{a, \gamma}(\R^3)}+C\ve^{-5/2}\|\phi\|_{L^2_{a, \gamma'}(\R^3)}.
\]
We also have $L_{w_\tau}(\chi\phi-\phi_1)=0$ and $\phi_1=\chi\phi=0$ along the surface $\ve M=r-\rho_\tau(z)$. Then, an estimate similar to  (\ref{decay estimate}), shows that actually $\chi\phi=\phi_1$. Similar argument applied in the set $-\ve M> r-\rho_\tau(z)$  ends the proof.
\end{proof}

\subsection{The deficiency space and the kernel of $L_{w_\tau}$}

Let us summarize our results so far. Let $g\in L^2_{a, \gamma}(\R^3)$ with $a^2+\gamma^2<\delta_\tau$, and cutoff functions $\chi^\pm(z)$ such that $\chi^+(z)+\chi^-(z)=1$ and $\mathrm{supp}\, \chi^+=[-1, \infty)$ be given.
\begin{itemize}
\item[(i)] As in  Proposition \ref{prop isomor} we can   solve 
\[
L_{w_\tau}\phi^\pm=   \chi^\pm g 
\]
where $\phi^\pm \in H^2_{-|a|, \gamma}(\R^3)_\pm$ (except for a finite set of $a$). 
\item[(ii)]
If we have $L_{w_\tau}\phi=g$, $\phi \in L^2_{a, \gamma'}(\R^3)$, $g\in L^2_{a, \gamma}(\R^3)$ with $\gamma>0$ and $\gamma'<\gamma$ then $\phi \in H^2_{a, \gamma}(\R^3)$. In particular if $g$ is decaying exponentially away from the surface $D_\tau$,  so that  we have $g\in L^2_{a, \gamma}(\R^3)\cap L^2_{a, -\gamma}(\R^3)$  then $\phi\in H^2_{a, \gamma}(\R^3)\cap H^2_{a, -\gamma}(\R^3)$. This means that the decay rate of the solution away from the nodal set improves together with the rate of decay of the right hand side. 

\item[(iii)] When the right hand side  decays both  along the nodal set and in the direction transversal to it, for example $g\in L_{a, \gamma}^2(\R^3)\cap  L^2_{a, -\gamma}(\R^3)$,  with $a>0$,  then we can use the parametrix to solve the equation $L_{w_\tau}\phi^+=\chi^+ g$ and determine a solution $\phi^+$ such that $\chi^+\phi^+\in  L^2_{a, \gamma}(\R^3)_+\cap  L^2_{a, -\gamma}(\R^3)_+ $. At the same time we can find another solution  $\phi_1^+$, such that    $\chi^+\phi^+_1\in L_{-a, \gamma}^2(\R^3)_+\cap  L^2_{-a, -\gamma}(\R^3)_+$ and we get the following decomposition:
\[
\phi^+=\sum_{j=1}^k Z^+_j+\phi^+_1,
\] 
where $Z^+_j$ are in the kernel of the operator $L_{w_\tau}$. Then we have
\[
\phi^+=\chi^+\phi^++\chi^-\phi^+=\chi^+\phi^++\chi^-\phi_1^++\sum_{j=1}^k Z^+_j\chi^-,
\]
where $(\chi^+\phi^++\chi^-\phi_1^+)\in  H^2_{a, \gamma}(\R^3)$. 
Of course we can argue similarly for the equation   $L_{w_\tau}\phi^-=\chi^- g$ and thus at the end we get the following formula
\[
\phi=\phi_0+\sum_{j=1}^k\chi^- Z^+_j+\sum_{j=1}^k \chi^+Z^-_j, \]
where $\phi_0\in H^2_{a, \gamma}(\R^3)$. This is the so called {\it linear decomposition formula}. It says that any solution to $L_{w_\tau}\phi = g$ can be decomposed into an exponentially decaying part and and a linear combination of $2k$  functions which are related to the residues of $\hat {L}^{-1}_{w_\tau,\gamma}(\zeta)$ at its poles. We say that these functions belong to the {\it deficiency space}. Clearly the elements of the kernel of $L_{w_\tau}$ (which is $k$ dimensional) belong to the deficiency space and thus removing them from it we obtain a space on which $L_{w_\tau}$ is an isomorphism (see Lemma (\ref{lemma isomor}) below). 
\end{itemize}  

Before stating precisely the next Lemma we introduce weighted Sobolev spaces
\[
\bar L_{a,\gamma}(\R^3):= L^2_{a, \gamma}(\R^3)\cap  L^2_{a, -\gamma}(\R^3), \quad \bar H^s_{a,\gamma}(\R^3):= H^s_{a, \gamma}(\R^3)\cap  H^s_{a, -\gamma}(\R^3).
\]
Note that $\phi\in  \bar L_{a,\gamma}(\R^3)$ decays away from $D_\tau$ as $\cosh^{-\gamma}\left(\frac{r-\rho_\tau(z)}{\ve}\right)$ if $\gamma>0$, and decays (for $a>0$) or grows (for $a<0$) along $D_\tau$ at the rate $\cosh^{-a} z$.  Based on observations (i)--(iii) we have:

\begin{lemma}\label{lemma isomor}
Let $\gamma>0$, $a > 0$,   with $a^2+\gamma^2<\delta_\tau$ and let us define the deficiency space
\[
\mathcal D_{w_\tau}=\mathrm{span} \left\{\chi^+ Z_j, \chi^- Z_j, j=1, \dots k, Z_j\in \mathrm{Ker}\, L_{w_\tau}\right\}.
\]
We further  decompose $\mathcal D_{w_\tau}=\mathcal K_{w_\tau}\oplus \mathcal E_{w_\tau}$, where $\mathcal K_{w_\tau}=\mathrm{Ker}\, L_{w_\tau}$. Then the operator  
\[
\begin{aligned}
L_{w_\tau}\colon \bar H^2_{a, \gamma}(\R^3)\oplus \mathcal E_{w_\tau} &\longrightarrow \bar L^2_{a, \gamma}(\R^3)\\
\phi&\longmapsto L_{w_\tau}\phi.
\end{aligned}
\]
is an isomorphism. 
\end{lemma}

}}

Note that $\mathrm{dim}\,\mathcal E_{w_\tau}=k=\mathrm{dim}\, \mathcal K_{w_\tau}$ and  that we know already that $k\geq  6=\mathrm{dim}\, \mathcal I_{w_\tau}$ where the linear subspace $\mathcal I_{w_\tau}$ was defined in (\ref{def kernel}).  We will show next that indeed $k=6$. 

\begin{proposition}\label{lemma kernel}
We have $\mathcal K_{w_\tau}=\mathcal I_{w_\tau}$.
\end{proposition} 
\begin{proof}[Proof of Proposition \ref{lemma kernel}]
The idea of the proof is to relate the kernel of the operator $L_{w_\tau}$  with the space of the Jacobi fields of the operator $\mathcal J_{D_\tau}$ that is explicitly known and in particular its dimension is $6$. Let us consider a  $\phi\in \mathcal K_{w_\tau}$. {\it A priori} it may happen that $\phi$ is exponentially increasing in the $z$ variable but we know already (see the argument leading to (\ref{est h2l2}) an also   Lemma \ref{decay in r} and Remark \ref{rem concentration estimate}) that it must be decaying at least like $\cosh^{-\gamma} (\frac{r-\rho_\tau(z)}{\ve})$ with some $\gamma>0$.  In particular all integrations  with respect to the transversal direction to $D_\tau$ that will appear below are justified.

Next, we note that  formula (\ref{expr v3}) suggests that near the surface $D_\tau$ the elements  of $ \mathcal K_{w_\tau}$ should be proportional, asymptotically as $\ve \to 0$,   to  ${\tt V}$ times a function on $D_\tau$. To make this rigorous we first prove  the following: 
\begin{lemma}\label{lemma 4.5}
Let $\phi\in \mathcal K_{w_\tau}$ be such that 
\begin{equation}\label{lem 4.5.1}
\int_\R (Y^*_{\ve}\phi)(\sy, \st)  {\tt V}(\sy,\st ) \chi_{\ve/\delta}(\st)\, d{\st}=0, \quad \forall \sy\in D_\tau. 
\end{equation}
Then we have $\phi\equiv 0$. 
\end{lemma}
\begin{proof}[Proof of Lemma \ref{lemma 4.5}]
As we have pointed out it is not hard to show that $\phi$ decays exponentially like $\cosh^{-\gamma}\big(\frac{r-\rho_\tau(z)}{\ve})$ and so we can compute
\[
\int_{\R^2} \phi^2({x}, z)\,d{x}=h(z), \qquad {x}=(x_1, x_2). 
\]
Direct calculation shows:
\begin{equation}
\label{h est}
\frac{\ve}{2}\frac{d^2 h}{d z^2}=\int_{\R^2}\big[\ve |\nabla_{x} \phi|^2-\frac{1}{\ve}f'(w_\tau) \phi^2\big]\,d{x}+\ve\int_{\R^2} |\partial_z \phi|^2\, d{x}.
\end{equation}
We claim that the orthogonality condition (\ref{lem 4.5.1}) implies 
\begin{equation}\label{alman}
\int_{\R^2}\big[\ve |\nabla_{x} \phi|^2-\frac{1}{\ve}f'(w_\tau) \phi^2\big]\,d{x}\geq \frac{\kappa}{\ve}\int_{\R^2} \phi^2({x}, z)\,d{x}=\frac{\kappa}{\ve} h,
\end{equation}
with some constant $\kappa>0$. To prove this claim we need:
\begin{lemma}\label{one d est}
There exists a constant $\kappa>0$ such that for any sufficiently large $R$ and  any  $v\in H^1((-R, R))$ it holds
\[
\int_{-R}^R |v'|^2 - f'(\varTheta) v^2\geq \kappa \int_{-R}^R v^2 \quad \mbox{whenever} \quad \int_{-R}^R v\varTheta'\chi_R=0,
\]
where $\chi_R$ is a smooth cutoff function supported in $(-R, R)$ such that $\chi_R(x)=1$ in $(-R/2,R/2)$. 
\end{lemma}
A proof of this Lemma (using  for instance (\ref{coerc zero form}) as a point of departure) is omitted. 

Changing  to Fermi coordinates  we have in $\mathcal N_\delta$: 
\begin{equation}
\label{grad fer}
\ve|\nabla_{{x}} \phi|^2=\frac{1}{\ve}|\partial_{\st}Y^*_{\ve}\phi|^2+\mathcal O(\ve)|\partial_{s}Y^*_{\ve}\phi|^2+\mathcal O(\ve)|\partial_\theta Y^*_{\ve}\phi|^2.
\end{equation}
Next, for a fixed $z$ we consider a diffeomorphism  $(x_1, x_2)\mapsto (\theta, \st)$ defined by
\[
x_j=\left(X(s,\theta)+\ve \st N_\tau(s,\theta)\right)\cdot {\tt e}_j
\]
where $s=s(\theta, \st;z)$ is determined from
\[
z= \left(X(s,\theta)+\ve \st N_\tau(s,\theta)\right)\cdot {\tt e}_3.
\]
The Jacobian matrix of this map can can be  calculated explicitly but for our purpose it is enough to note that 
\[
dx_1 dx_2=\ve \mu_0(\theta)\,d\theta d\st+\ve^2\st\mu_1(\theta, \st)\,d\theta d\st,
\]
where $\mu_0$, $\mu_1$ are positive densities  and 
\[
|\mu_1(\theta, \st)|\leq C.
\]
From this we find
\begin{equation}\label{3int}
\begin{aligned}
\int_{\R^2}\big[\ve |\nabla_{x} \phi|^2-\frac{1}{\ve}f'(w_\tau) \phi^2\big]\,d{x}&\geq \int_0^{2\pi} \left\{\int_{|\st|\leq \delta/\ve}\left[|\partial_{\st}Y^*_{\ve}\phi|^2-f'\left(\varTheta(\st)\right)|Y^*_{\ve}\phi|^2\right]\,d\st\right\}\,\mu_0 d\theta\\
&\qquad +\int_{\R^2\setminus \mathcal N_\delta}\left[\ve |\nabla_{x} \phi|^2+\frac{1}{\ve} \phi^2\right]\,d{x}\\
&\qquad 
-K(\delta+\ve)\int_{\R^2}\left[\ve |\nabla_{x} \phi|^2+\frac{1}{\ve} \phi^2\right]\,d{x}
\end{aligned}
\end{equation}
The potential $f'(w_\tau)$    in the first line on the left can be replaced by $f'(\varTheta)$ on the right of this line since $Y^*_\ve w_\tau=\varTheta+\mathcal O(\ve)$.  The term in the second line above appears because $f'(w_\tau)<-2+\eta$ in the complement of $\mathcal N_\delta$. Finally, all the other terms are of smaller size and can be controlled by the integral in the third line times $K(\delta+\ve)$, where $K$ is a constant.  Using Lemma \ref{one d est} and going back to the original variables we get
\[
\int_0^{2\pi} \left\{\int_{|\st|\leq \delta/\ve}\left[|\partial_{\st}Y^*_{\ve}\phi|^2-f'\left(\varTheta(\st)\right)|Y^*_{\ve}\phi|^2\right]\,d\st\right\}\,\mu_0 d\theta \geq \frac{C}{\ve} \int_{\mathcal N_\delta}\phi^2\,dx.
\]
It follows 
\[
\int_{\R^2}\big[\ve |\nabla_{x} \phi|^2-\frac{1}{\ve}f'(w_\tau) \phi^2\big]\,d{x}
\geq  \frac{C}{\ve} \int_{\R^2}\phi^2\,dx -K(\ve+\delta)  \int_{\R^2}\left[\ve |\nabla_{x} \phi|^2+\frac{1}{\ve} \phi^2\right]\,d{x},
\]
hence
\[
[1+K(\ve+\delta)]\int_{\R^2}\big[\ve |\nabla_{x} \phi|^2-\frac{1}{\ve}f'(w_\tau) \phi^2\big]\,d{x}
\geq  \frac{C}{\ve} \int_{\R^2}\phi^2\,dx-\frac{K(\ve+\delta)}{\ve}\int_{\R^2}[1+f'(w_\tau)] \phi^2\,dx
\]
which gives (\ref{alman}) provided that $\ve$ and $\delta$ are small enough.

From (\ref{alman}) and (\ref{h est}) we find
\[
\frac{\ve}{2}\frac{d^2 h}{d z^2}-\frac{\kappa}{\ve}h>0.
\]
By Lemma \ref{lemma isomor} we know {\it a priori} that $\phi$, hence $h$, is growing  in $z$ at $\pm\infty$  at some exponential rate which is independent on $\ve$. Applying  the comparison principle   we see  that $h$, and hence $\phi$, is actually decaying as $z\to \pm \infty$,  at some exponential rate proportional to $\ve^{-1}$.  Using again  orthogonality condition (\ref{lem 4.5.1}) we  calculate 
\[
\langle-L_{w_\tau} \phi, \phi\rangle=\int_{\R^3} \left[\ve |\nabla \phi|^2-\frac{1}{\ve}f'(w_\tau) \phi^2\right]\, d{x} dz\geq c\ve^{-1}\|\phi\|^2_{L^2(\R^3)},
\]
hence $\phi\equiv 0$ as claimed. {\blue{This ends the proof of the Lemma.}}

\end{proof}

We continue with the proof of the Proposition. For a given  $\phi\in\mathcal K_{w_\tau} $  we  define 
\[
\varphi=  (Y^*_\ve \phi)\chi_{\delta/\ve}.
\]
The function $\varphi$ is a cutoff of $Y^*_\ve \phi$ and is supported  in $\mathcal N_{\delta}$. Since $\phi\in\bar H^2_{a, \gamma}(\R^3)$ with some $a\in \R$, and $\gamma>0$ both small  ($\phi$ decays or grows in $z$ like $\cosh^{-a} z$,  and it decays like $\cosh^{-\gamma}\big(\frac{r-\rho_\tau(z)}{\ve}\big)$ away from $D_\tau$) we have that $\varphi\in H^2_{a_*, \gamma_*}(D_\tau \times \R)$ with some $a_*\in \R$ and $\gamma_*>0$ both small. We also have
\[
\|\varphi\|_{\bar L^2_{a_*, \gamma_*}(D_\tau\times \R)}\leq C\ve^{-1/2}\|\phi\|_{\bar L^2_{a, \gamma}(\R^3)},
\]
with similar estimates for other Sobolev norms. Note that since $\phi$  decays like $\cosh^{-\gamma}\big(\frac{r-\rho_\tau(z)}{\ve}\big)$ away from $D_\tau$ then $\varphi$ decays at least like $\cosh^{-\bar \gamma} \st$ with some $\bar\gamma >0$. Above estimate holds then for any  $\gamma_*<\bar \gamma$ and  we will consider only $\gamma_*$ restricted this way.

In what follows we will argue by contradiction and we will assume that $\mathrm{dim}\, \mathcal K_{w_\tau}>6$. Since we know explicitly six linearly independent elements in $\mathcal K_{w_\tau}$, which are the geometric Jacobi fields spanning the subspace $\mathcal I_{w_\tau}$ defined in (\ref{def kernel}) we can find a function  $\phi\in \mathcal K_{w_\tau}$ such that $\phi\notin \mathcal I_{w_\tau}$ and in particular we can assume 
\begin{equation}
\label{orto phi}
\int_{D_\tau\times \R}\chi_{\delta/\ve} (Y^*_\ve \phi) (Y^*_\ve \Phi_\bullet)\cosh^{a_*}(s)\,dV_{D_\tau}d\st=0, \qquad \forall \Phi_\bullet \in \mathcal I_{w_\tau}.
\end{equation}

We decompose 
\[
\varphi=\psi{\tt V} +\varphi^\parallel, \qquad \int_\R \chi_{\ve/\delta}\varphi^\parallel(\sy,\st) {\tt V}(\sy,\st )\, d\st=0.
\]
From Lemma \ref{lemma 4.5} we know that $\psi\neq 0$ and therefore we can assume $\|\psi\|_{L^2_{a^*}(D_\tau)}=1$ (indeed we expect  
$\|\varphi^\parallel\|_{L^2_{a_*, \gamma_*}(D_\tau\times \R)}=o(1)$).  We compute
\[
\mathbb L_{w_\tau}\varphi=(Y^*_\ve L_{w_\tau})\varphi+[\mathbb L_{w_\tau}-(Y^*_\ve L_{w_\tau})]\varphi\equiv g,
\]
where, more explicitly,
\[
\begin{aligned}
(Y^*_\ve L_{w_\tau})\varphi&=\ve^{-1}[(Y^*_\ve \phi)\partial_{\st\st}\chi_{\ve/\delta}+2\partial_\st (Y^*_\ve \phi)\partial_\st \chi_{\ve/\delta}]-(H_{D_\tau}+\ve \st |A_{D_\tau}|^2 +\mathbb Q_\ve) (Y^*_\ve \phi)\partial_\st \chi_{\ve/\delta}\\
[\mathbb L_{w_\tau}-(Y^*_\ve L_{w_\tau})]\varphi&=(1-\chi_{\ve/\delta})(\ve\st |A_{D_\tau}|^2+\mathbb Q_\ve)\partial_\st \varphi-\ve (1-\chi_{\ve/\delta})\mathbb A_\ve \varphi.
\end{aligned}
\]

It is not hard to see that
\[
\|\chi_{\ve/\delta}g\|_{\bar L^2_{a_*, \gamma_*} (D_\tau\times \R)}\leq \mathcal O(e^{\,-c\delta/\ve})\|\varphi\|_{\bar H^2_{a_*, \gamma_*} (D_\tau\times \R)},
\]
since $\gamma_*<\bar \gamma$.
Using this we can  calculate
\[
\int_\R \chi_{\ve/\delta} \mathbb L_{w_\tau}\varphi {\tt V}\,d\st=\int_\R \chi_{\ve/\delta} g\varphi {\tt V}\,d\st
\]
which gives
\begin{equation}
\label{eq for psi}
\mathcal J_{D_\tau} \psi=T(\varphi^\parallel, \psi), 
\end{equation}
where $T$ is a linear operator satisfying
\begin{equation}
\label{op t}
\|T(\varphi^\parallel, \psi)\|_{L^2_{a_*}(D_\tau)}\leq C\ve^{-1}\|\varphi^\parallel\|_{\bar L^2_{a_*, \gamma_*}(D_\tau \times \R)}+C\ve\|\varphi^\parallel\|_{\bar H^2_{a_*, \gamma_*}(D_\tau \times \R)} +C\ve^{1-\alpha}\|\psi\|_{H^1_{a_*}(D_\tau)}+C\delta\|\psi\|_{H^2_{a_*}(D_\tau)},
\end{equation}  
with some $\alpha\in (0,1)$. Next we will estimate $\varphi^\parallel$. Since this argument is similar to that of Proposition \ref{prop isomor} we will outline the main points omitting some tedious but  straightforward calculations. Let $K>0$ be a large constant and $\chi^\pm\colon \R\to \R_+$ be smooth cutoff functions such that $\chi^++\chi^-\equiv 1$, $\chi^+(s)=1$ when $s>1$ and $\chi^+(s)=0$ when $s<-K$ and additionally $K|\chi^\pm_s|+K^2|\chi^\pm_{ss}|\leq C$. 

We define $\varphi^{\parallel, \pm}=\chi^\pm \varphi^\parallel$. Taking the Fourier-Laplace transform (with respect to $s$)  we get
\[
\begin{aligned}
{(\mathbb L_{w_\tau}\varphi^{\parallel, \pm})}^\wedge&=(\chi^\pm \mathbb L_{w_\tau}\varphi^{\parallel})^\wedge+\left(\left[\mathbb L_{w_\tau}, \chi^\pm\right]\varphi^\parallel\right)^\wedge\\
&=(\chi^\pm g)^\wedge-\left(\chi^\pm \mathbb L_{w_\tau}(\psi {\tt V})\right)^\wedge+\left(\left[\mathbb L_{w_\tau}, \chi^\pm\right]\varphi^\parallel\right)^\wedge.
\end{aligned}
\]
We can  project
\begin{equation}
\label{equ 102}
\begin{aligned}
\int_{[0,T_\tau]\times [0, 2\pi]\times  \R}  \hat \varphi^{\parallel,\pm} {(\mathbb L_{w_\tau}\varphi^{\parallel, \pm})}^\wedge&=
\int_{[0,T_\tau]\times [0, 2\pi]\times\R} \hat \varphi^{\parallel,\pm}(\chi^\pm g)^\wedge\\
&\qquad -\int_{[0,T_\tau]\times[0, 2\pi]\times \R} \hat \varphi^{\parallel,\pm}\left(\chi^\pm \mathbb L_{w_\tau}(\psi {\tt V})\right)^\wedge\\
&\qquad +\int_{[0,T_\tau]\times [0, 2\pi]\times\R} \hat \varphi^{\parallel,\pm}\left(\left[\mathbb L_{w_\tau}, \chi^\pm\right]\varphi^\parallel\right)^\wedge
\end{aligned}
\end{equation}
Since we have
\[
\int_\R  \chi_{\ve/\delta} \hat \varphi^{\parallel,\pm} {\tt V}\,d\st=0,
\]
therefore the bilinear form on the left hand side in (\ref{equ 102}) is positive definite and  by an argument similar to the one in Proposition \ref{prop isomor} we get
\[
\left|\int_{[0,T_\tau]\times [0, 2\pi]\times  \R}  \hat \varphi^{\parallel,\pm} {(\mathbb L_{w_\tau}\varphi^{\parallel, \pm})}^\wedge\right| \geq \frac{C}{\ve}\|\hat \varphi^{\parallel,\pm} \|^2_{L^2([0,T_\tau]\times [0, 2\pi]\times  \R)}\geq \frac{C}{\ve}\|\varphi^{\parallel,\pm} \|^2_{L^2_{a_*}(D_\tau\times \R)_\pm},
\]
where the last inequality follows from Plancherel's identity. 
Using Cauchy-Schwarz inequality and Plancherel identity again on the right hand side of (\ref{equ 102}) we find 
\[
\ve^{-1}\|\varphi^{\parallel,\pm} \|_{L^2_{a_*}(D_\tau\times \R)_\pm}\leq C\left( \|\chi^\pm g\|_{L^2_{a_*}(D_\tau\times \R)_\pm}+\|\chi^\pm \mathbb L_{w_\tau}(\psi {\tt V})\|_{L^2_{a_*}(D_\tau\times \R)_\pm}+\left\|\left[\mathbb L_{w_\tau}, \chi^\pm\right]\varphi^\parallel\right\|_{L^2_{a_*}(D_\tau\times \R)_\pm}\right)
\]
Using an argument similar to the one indicated in Remark \ref{rem concentration estimate} and Remark \ref{bootstrap} we can show from this
\begin{equation}
\label{equ 103}
\begin{aligned}
&\ve^{-1}\|\varphi^{\parallel,\pm} \|_{L^2_{a_*,\gamma_*}(D_\tau\times \R)_\pm}+\ve \|\nabla \varphi^{\parallel,\pm} \|_{L^2_{a_*,\gamma_*}(D_\tau\times \R)_\pm}+\ve\|D^2\varphi^{\parallel,\pm} \|_{L^2_{a_*,\gamma_*}(D_\tau\times \R)_\pm}\leq C R\\
&R\equiv\left( \|\chi^\pm g\|_{L^2_{a_*, \gamma_*}(D_\tau\times \R)_\pm}+\|\chi^\pm \mathbb L_{w_\tau}(\psi {\tt V})\|_{L^2_{a_*,\gamma_*}(D_\tau\times \R)_\pm}+\left\|\left[\mathbb L_{w_\tau}, \chi^\pm\right]\varphi^\parallel\right\|_{L^2_{a_*,\gamma_*}(D_\tau\times \R)_\pm}\right).
\end{aligned}
\end{equation}
We have 
\[
\begin{aligned}
\|\chi^\pm g\|_{L^2_{a_*, \gamma_*}(D_\tau\times \R)_\pm}&\leq \mathcal O(e^{\,-c\delta/\ve})\left(\|\varphi^\parallel\|_{H^2_{a_*, \gamma_*}(D_\tau\times \R)}+
\|\psi\|_{H^2_{a_*}(D_\tau)}\right)\\
\|\chi^\pm \mathbb L_{w_\tau}(\psi {\tt V})\|_{L^2_{a_*,\gamma_*}(D_\tau\times \R)_\pm}&\leq C\ve\|\psi\|_{H^2_{a_*}(D_\tau)}
\\
\left\|\left[\mathbb L_{w_\tau}, \chi^\pm\right]\varphi^\parallel\right\|_{L^2_{a_*,\gamma_*}(D_\tau\times \R)_\pm}&\leq \frac{C\ve }{K}\|\varphi^\parallel\|_{H^1_{a_*,\gamma_*}(D_\tau\times \R)}.
\end{aligned}
\]
Combining these inequalities we get from (\ref{equ 103})
\begin{equation}
\label{equ 104}
\ve^{-1}\|\varphi^{\parallel} \|_{L^2_{a_*}(D_\tau\times \R)}+\ve \|\nabla \varphi^{\parallel} \|_{L^2_{a_*,\gamma_*}(D_\tau\times \R)}+\ve\|D^2\varphi^{\parallel} \|_{L^2_{a_*,\gamma_*}(D_\tau\times \R)}\leq C\ve \|\psi\|_{H^2_{a_*}(D_\tau)}.
\end{equation}
This and estimate (\ref{op t}) imply
\[
\|T(\varphi^\parallel, \psi)\|_{L^2_{a_*}(D_\tau)}\leq C(\ve^{1-\alpha}+\ve) \|\psi\|_{H^2_{a_*}(D_\tau)}+C\delta\|\psi\|_{H^2_{a_*}(D_\tau)}.
\]
Decomposing $\psi=\psi^++\psi^-$, where $\psi^\pm=\chi^\pm\psi$ we can use the Fourier-Laplace transform to show that 
\[
\psi=\psi_0+\psi_1
\]
where $\psi_0$ is a linear combination of the the geometric Jacobi fields and 
\begin{equation}
\label{psi111}
\|\psi_1\|_{H^2_{a_*}(D_\tau)}\leq C\|T(\varphi^\parallel, \psi)\|_{L^2_{a_*}(D_\tau)}\leq C(\ve^{1-\alpha}+\ve) \|\psi\|_{H^2_{a_*}(D_\tau)}+C\delta\|\psi\|_{H^2_{a_*}(D_\tau)}.
\end{equation}
At the same time
from (\ref{orto phi}),  (\ref{equ 104}) and Lemma \ref{lemma asymptot kernel} we see  that $\psi$ satisfies
\[
\int_{D_\tau}\psi \Phi_\tau ^\bullet \cosh^{a_*} s\,dV_{D_\tau}=0\Longrightarrow \int_{D_\tau}\psi_0 \Phi_\tau ^\bullet \cosh^{a_*} s\,dV_{D_\tau}=-\int_{D_\tau}\psi_1 \Phi_\tau ^\bullet \cosh^{a_*} s\,dV_{D_\tau}
\]
for each geometric Jacobi field $\Phi_\tau^\bullet$ of $J_{D_\tau}$. It follows that 
\[
\|\psi_0\|_{L^2_{a_*}(D_\tau)}\leq C\|\psi_1\|_{L^2_{a_*}(D_\tau)}
\]
which, together with (\ref{psi111}), implies $\psi_1\equiv 0$ hence $\psi\equiv 0$, which is a contradiction. 
The proof of the proposition is complete.

\end{proof}


\providecommand{\bysame}{\leavevmode\hbox to3em{\hrulefill}\thinspace}
\providecommand{\MR}{\relax\ifhmode\unskip\space\fi MR }
\providecommand{\MRhref}[2]{%
  \href{http://www.ams.org/mathscinet-getitem?mr=#1}{#2}
}
\providecommand{\href}[2]{#2}

\end{document}